\documentclass[12pt]{amsart}
\usepackage[dvipsnames]{xcolor}
\usepackage{amssymb}
\usepackage{amsmath}
\newtheorem{thm}{Theorem}[section]
\newtheorem{prop}[thm]{Proposition}
\newtheorem{cor}[thm]{Corollary}
\newtheorem{lem}[thm]{Lemma}
\newtheorem{rema}[thm]{Remark}

\newtheorem{defn}[thm]{Definition}
\theoremstyle{remark}
\newtheorem{rem}[thm]{Remark}
\numberwithin{equation}{section}

\newcommand{\h}[0]{\mathfrak{h}}

\def \Z{\mathbb Z}
\def \C{\mathbb C}
\def \Q{\mathbb Q}
\def \N{\mathbb N}

\def \End{{\rm End}}

\def \mod{{\rm mod}}

\def \l{\lambda}

\newcommand{\1}{\bf 1}
\def \<{\langle}
\def \>{\rangle}

\def \l{\lambda }

\newcommand{\bea}{\begin{eqnarray}}
\newcommand{\eea}{\end{eqnarray}}
\newcommand{\be}{\begin {equation}}
\newcommand{\ee}{\end{equation}}

\newcommand{\W}{\mathcal W}

\newcommand{\la}{\langle}
\newcommand{\ra}{\rangle}

\def \br{\begin{rem}\label}
	\def \er{\end{rem}}

\begin{document}

\title[Screening operators for vertex operator algebras]{Classification of screening systems for lattice vertex operator algebras}
\author{Katrina Barron}
\address{Department of Mathematics, University of Notre Dame}
\email{kbarron@nd.edu}
\thanks{K. Barron was supported by Simons Foundation Collaboration Grant 282095}
\author{Nathan Vander Werf}
\address{Department of Mathematics, University of Nebraska at Kearny}
\email{vanderwerfnp@unk.edu}

\date{December 11, 2018}

\begin{abstract}
We study and classify systems of certain screening operators arising in a generalized vertex operator algebra, or more generally an abelian intertwining algebra with an associated vertex operator (super)algebra.   Screening pairs arising from weight one primary vectors acting commutatively on a lattice vertex operator algebra (the vacuum module) are classified into four general types, one type of which has been shown to play an important role in the construction and study of certain important families of $\mathcal{W}$-vertex algebras.   These types of screening pairs we go on to study in detail through the notion of a system of screeners, which are lattice elements or ``screening momenta" which give rise to screening pairs.  We classify screening systems for all positive definite integral lattices of rank two, and for all positive definite even lattices of arbitrary rank when these lattices are generated by a screening system.
\end{abstract} 

\maketitle

\section{Introduction} 
Up until about ten years ago, the main focus of the theory of vertex operator algebras, or more generally vertex operator superalgebras, had been on understanding rational vertex operator algebras, that is, vertex operator algebras whose admissible (or $\N$-graded) module category is of ``finite representation type", meaning that there are only finitely many indecomposable modules, and semi-simple, so that all indecomposable modules are irreducible.  The progress prior to 2006 led to several important breakthroughs, for instance Zhu's associative algebra and Zhu's modular invariance theorem \cite{Zhu}.  One of the fundamental results of \cite{Zhu} is the establishment of a correspondence between the irreducible $\mathbb{N}$-gradable modules of $V$ and Zhu's associative algebra $A(V)$.  In the same paper, Zhu also introduced a certain finiteness condition on $V$, called the $C_2$-cofinite condition, which was necessary for an important modular invariance result. It is now known that when $V$ satisfies the $C_2$-cofinite property, every $V$-module is $\N$-gradable, and $V$ has only finitely many irreducible modules. If $V$ is both $C_2$-cofinite and rational, then the category of $V$-modules is in fact equivalent to the category of $A(V)$-modules.  Up until 2007, it was conjectured that the $C_2$-cofinite condition and rationality were equivalent. This is now known not to be the case; see for instance \cite{Abe}.

Prior to this, irrational vertex operator algebras did not attract as much attention since it was thought that they would have too many irreducible objects, and thus they would not have interesting and important  features such as modular invariance properties.  This particular feature has a remedy though if we consider irrational vertex operator algebras of finite representation type, such as the irrational $C_2$-cofinite vertex operator algebras, in which there is a modified version of Zhu's original modular invariance theorem \cite{Miy}.  For this reason,  $C_2$-cofinite vertex operator algebras appear to be very special objects, the irrational ones now being the main object of study for logarithmic conformal field theory.  

Using Zhu's algebra, one can make a rough analogy between the category of $C_2$-cofinite vertex operator algebras and the category of finite dimensional associative algebras with unity, where the rational vertex operator algebras correspond to the semisimple subcategory.  But this leads to something of a paradox:  On the algebra side, for dimension over $\mathbb{C}$ larger than one, up to isomorphism, there are more (even infinitely more, e.g., in dimension greater than 4) nonsemisimple associative algebras with unity than there are semisimple ones.  But on the vertex operator algebra side, only a small handful of cases of nonrational $C_2$-cofinite vertex operator algebras are known at the moment, with even less understood about the structure of their representations, making it important to study the known examples and seek new models.

This paper, along with \cite{BV-W-algebras}, aims to start laying the foundations to systematically fill in some of the gaps of our class of examples and further our understanding of the phenomenon of irrational $C_2$-cofinite vertex operator algebras.  In particular, in this paper we give a systematic treatment of the notion of screening operators and screening pairs in the setting of generalized vertex operator algebras and abelian intertwining algebras associated to positive definite integral lattices.   We introduce the notion of ``screeners" which are the lattice elements $\alpha$ which give rise to screening pairs $(e^{-\alpha/p}_0, e^{\alpha/p'}_0)$, and correspond to special type of lattice element, or in physics terms, the ``screening momenta" which give rise to a screening pair.  We classify such screeners for all rank two positive definite integral lattices, and all positive definite even lattices that are generated by a screening system.

Screening pairs give crucial information that can be used in analyzing the structure of the vertex operator algebra realized by either the kernel of a screening operator or  arising itself as the intersection of the kernels of a pair, or tuple of screening operators.  In particular, screening operators and more generally screening pairs, appear to be important ingredients in the future construction and study of $C_2$-cofinite vertex operator algebras.

This ground work allows for future work in looking at kernels of screening operators and thus 
generalizing a certain class of $\mathcal{W}$-algebras that are one of the known cases of irrational $C_2$-cofinite vertex operator algebras, the so called  $\mathcal{W}(p,p')$-algebras. These $\mathcal{W}(p,p')$ algebras are generalizations of the $\mathcal{W}(p) = \mathcal{W}(p,1)$ algebras first studied in \cite{Ka1}, \cite{GK1}, \cite{GR}, with the case of symplectic fermions which correspond to $\mathcal{W}(2)$ (\cite{Ka2}) shown to be irrational $C_2$-cofinite in \cite{Abe}.

Screening operators for vertex operator algebras associated to a lattice  (the so called  free field realization approach to constructing new models)  was developed in \cite{Wakimoto}, \cite{Fel}, \cite{FFr}, although screening operators appeared in the physics literature earlier in, e.g., \cite{DF1} and \cite{DF2}. In the lattice setting screening operators are viewed as special intertwining operators in the Feign-Fuchs module theory.  The $\mathcal{W}(p)$ algebra was studied as the kernel of screening operators (rather than through the cohomology of the screening \cite{Fel}) in \cite{FFHST}, and  this family of $\mathcal{W}$-algebras has since been the most studied class of logarithmic models,  e.g.  in addition to the works cited above, in  \cite{am1}, \cite{Am5}, \cite{am2}, \cite{am10}, \cite{am4}, \cite{FGST1}--\cite{FGST3}, \cite{NT}, \cite{TW1}, \cite{TW2}, among others.  Generalizations to  $\mathcal{W}(p,p')$ algebras realized via kernels of screening operators were introduced in \cite{FGST1} and examples of these have been further studied in, for instance, \cite{AdM}, \cite{Am5}, \cite{am-imrn}, \cite{am8}, \cite{am9}, in addition to works cited above.  Higher rank cases are discussed in, for instance, \cite{FT} and \cite{am-imrn}.  For these and other related developments we refer the reader to \cite{EFH}, \cite{FFHST}, \cite{FHST} \cite{FGST1}--\cite{FGST3}, \cite{F1}, \cite{FG}, \cite{GK1}, \cite{GK2}, \cite{GR}, \cite{Ka1}, \cite{Ka2}, \cite{KaW}, \cite{F2} and \cite{Ga}.

In addition, there is a growing body of work devoted to various interactions between the $\mathcal{W}(p)$-algebras  and the quantum group $\bar{\mathcal{U}}_q(sl_2)$, or its variations, at a root of unity, first observed in \cite{FGST1}.  In \cite{FGST1}, it was proved for $p = 2$ and conjectured for general $p$, that as $\mathbb{C}$-linear categories, the category of representations for $\bar{\mathcal{U}}_q(sl_2)$ is equivalent to the category of representations for $\mathcal{W}(p)$.  This was then proved in the general $p$ case by \cite{NT}.  However the category of representations for $\mathcal{W}(p)$  is a braided tensor category, whereas that of $\bar{\mathcal{U}}_q(sl_2)$ is not.  Modifying the usual coproduct structure of $\bar{\mathcal{U}}_q(sl_2)$, an equivalence of ribbon categories is conjectured \cite{CGR}.  More generally, it is expected that this is just one instance of a broader connection between $\mathcal{W}$-algebras and quantum groups, e.g., \cite{Lentner}.

Using the rank one lattice embedding of $W(p,p')$ developed in \cite{am1}, we wish to generalize $\mathcal{W} (p,p')$ to a larger family of vertex operator algebras that are embedded in certain higher rank lattices. In general they appear to be just the right size in order to satisfy the $C_2$-cofinite condition, and in some of the examples we construct in a subsequent paper \cite{BV-W-algebras} (following and expanding upon \cite{V-thesis}), we prove the $C_2$-cofinite condition, as well as analyze other aspects of their structure.  In  \cite{BVY}, the authors along with Yang, show how the higher level Zhu's algebras as defined in \cite{DLM} can be used to understand the indecomposable objects for an irrational vertex operator algebra, and this work will be used to study the constructions arising from the screening operators studied and classified in this paper.

As indicated above, $\mathcal{W}$-algebras are some of the most exciting objects in representation theory, and have been studied intensively by mathematicians and physicists over the last two decades.  There are several different types of $\mathcal{W}$-algebras in the literature; see for instance \cite{FB}, \cite{BT}, \cite{W2}.   Here, for context to the motivations and main applications of this results of this paper, we give a definition of $\mathcal{W}$-algebras, or as we formally call them, {\it $\mathcal{W}$-vertex (super)algebras}, slightly modified from the definition in \cite{am4}:

\begin{defn} \label{W-algebra}
A {\em $\mathcal{W}$-vertex (super)algebra}, or simply {\em $\mathcal{W}$-algebra}, $V$ is a vertex (super)algebra strongly generated by a finite set of quasi-primary vectors $\omega_1,\dots ,\omega_k$ of weight $2$, and a finite set of primary vectors $u^{1},\dots,u^{r}$ for $r,k\in \Z_{+}$, such that the Virasoro element which grades $V$ by $L(0)$-eigenvalues is given by $\omega= \displaystyle\sum_{i=1}^{k} \omega_i$.
\end{defn} 

Here {\it strongly generated} means that elements of the form 
\begin{align}
v^{1}_{-j_1}\cdots v^{m}_{-j_m}\mathbf{1},\quad j_1,\dots j_m \geq 1
\end{align} 
span $V$ as a vector space with $v^{j}= u^i$ or $v^j= \omega_i$ for some $i$. We say $v$ is a {\em primary} vector if both $L(n)v=0$ for $n>0$ and $v$ is an eigenvector for $L(0)$. We say $v$ is {\em quasi-primary} if only $L(1)v=0$.  

The only change in the definition compared to that given in \cite{am4} is that we now allow for there to be more than one quasi-primary weight two vector in the strongly generating set.  This slight modification seems to be better suited in the higher rank case as it allows the tensor product of  $\mathcal{W}$-algebras (and certain extensions by additional primary vectors) to again be a $\mathcal{W}$-algebra.  

If $\mbox{wt}(u_i)=l_i\in \Z$, we say that $V$ is of {\it type $\mathcal{W}(2^r,l_1,\dots,l_k)$}. 

\begin{defn}
We call a vertex (super)algebra $V$ a {\em weak $\mathcal{W}$-algebra} if the strongly generating set not only consists of the vectors $w_1,\dots,w_r$ and $u^1,\dots,u^k$ in Definition \ref{W-algebra}, but also possibly some weight one vectors $v^1,\dots,v^s$ satisfying $L(2)v^l=0$, $L(1)v^l\neq 0$ and 
\[\omega= \sum_{j=1}^{r} \omega_j+ \sum_{l=1}^s (v^l)_{-1}v^l+a_l v^l_{-2}\mathbf{1}\]
 for certain $a_l\in \C$. 
\end{defn} 
 Then for instance, the lattice vertex operator superalgebra $V_L$ for $L$ a positive definite integral lattice is an example of a weak $\mathcal{W}$-algebra under this definition. 

The $\mathcal{W}(p)$-algebras are examples of $\mathcal{W}$-algebras and are known by several terms in the literature, including: The triplet vertex algebra (due to the presence of 3 primary vectors); the $\mathcal{W}(2,(2p-1)^3)$-algebra; the $\mathcal{W}(p,1)$-algebra (since the central charge is $c_{p,1}=1-6((p-1)^2/p)$).   Denoting this family of $\mathcal{W}$-algebras by $\mathcal{W}(p)$ for the positive integer $p$, also corresponds to the fact that $\mathcal{W}(p)$ is a subalgebra of the rank one lattice with generator having norm squared $2p$, which was discovered in \cite{am1}, where Adamovi\'{c} and Milas also proved it satisfied the $C_2$-cofinite condition and the irreducible modules were classified.

Some of the constructions and ideas used and discussed earlier for the $\mathcal{W}(p)$-algebra can be generalized to the $\mathcal{W}(p,p')$ models as well as higher rank analogues.   As stated earlier, the motivation of this paper is to lay the foundations for the construction and classification of higher rank generalizations of a such $\mathcal{W}$-algebras.

In this paper, first in Section 2, we review the notions of abelian intertwining algebra and vertex algebra as well as certain types of modules, following for example  \cite{DL1}, \cite{DL}, \cite{Li}. 

In Section 3, we give axiomatic definitions of the notion of screening operator and screening pair for an abelian intertwining algebra, and prove some basic properties.   We will denote a screening pair to be a pair of screening operators arising from weight one primary vectors, which commute and satisfy a certain additional property.  Such pairs will then be the focus of our study.    We note that the types of screening operators we focus on are those that arise as graded derivations in an abelian intertwining algebra $\tilde{V}$ which act on a vertex operator (super)algebra $V$ which is a subspace of $\tilde{V}$, sometimes called the ``vacuum module".    However, there are more general notions of screening operator, cf. \cite{Lentner}.

In Section 4, we restrict to vertex operator (super)algebras $V_L$ associated to a positive definite integral lattice $L$, first reviewing this notion, following \cite{LL}, \cite{BV1}, \cite{Xu}.   We classify screening pairs in $V_L$ coming from weight one primary vectors acting on the vacuum module $V_L$ into four types.  

In Section 5, we study the screening pairs which are the type that arise in the group algebra $1 \otimes \mathbb{C}[L]$ embedded in $V_L$, which in turn are determined by certain elements in the underlying lattice $L$ which we call {\it screeners}.  These screeners are certain lattice elements which give rise to the screening screening pairs, and thus in physics terminology, it would be appropriate to call these screeners, ``screening momenta".  

In Section 5, we go on to show that every screening pair of type two in our previous classification from Section 4, gives rise to a screener, and every screener of a certain type gives rise to a screening pair.  Thus to study screening pairs of this type, it suffices to study screeners, or in particular a linearly independent subset of the set of screeners, called a screening system, which generates a sublattice of $L$.  Such screening systems are of particular interest when the cardinality of the screening system coincides with the rank of the lattice.   

In fact, one of the main motivations of this work is to provide a setting and lay the foundation for studying the intersection of the kernel of multiple screening operators in $V_L$ in the case $L$ is of higher rank. This leads to many new vertex operator subalgebras of $V_L$ for certain $L$ that we predict are likely $C_2$-cofinite and irrational.  

In general it is difficult to classify screeners for a general positive definite integral lattice.  But we completely classify screeners in the following two cases: When the lattice is of rank two; and when $L$ is a positive definite even lattice generated by a screening system.  In the latter case, we prove that when $L$ is generated by a screening system then $L$ is an orthogonal direct sum of lattices that are rescaled simply laced root lattices.  

In Section 5.1, we classify all screeners for rank two positive definite integral lattices, and in Section 5.2, we classify all screeners in a simply laced root lattice.  Finally in Section 5.3, we summarize our results and classify all possible screeners for a positive definite even lattice of arbitrary rank generated by a screening system, by classifying all screeners that arise within orthogonal components (those arising from simply laced root systems) and across orthogonal components.   In particular, we show the following:

\begin{thm}
If $L$ is generated by screeners (i.e., momenta of screening pairs), then $L$ is the orthogonal direct sum of lattices of the form $\sqrt{p} K$, for $p$ a positive integer and $K$ a simply laced root lattice.  Furthermore the full set of screeners for each of the orthogonal components of the simply laced root lattices are as follows:

  For type $A_n$,  for $n >3$, and for $E_n$, for $n = 6, 7, 8$, the set of screeners is just the set of roots.  
  
  For $A_2$ and $A_3$, there are 6 additional screeners that are not roots which give rise to the root lattices for $G_2$ and $C_3$, respectively. 
  
 For the $D$-type lattices, there are 24 additional screeners that are not roots for $D_4$ giving rise to $F_4$, and for $n>4$, the lattice $D_n$ has $2n$ nonroot screeners giving the roots of $C_n$. 
 
 Finally, a screening system of $n$ orthogonal copies of $A_1$, for $n \geq 2$, has as its set of screeners the roots of $B_n$ (dual to $C_n$). 
\end{thm}

Note that this result implies that in some sense, the set of screeners in the orthogonal components of a positive definite even lattice generated by screeners, is maximal in terms of the possible root system.  \\

{\bf Acknowledgements:} This paper consists, essentially, of the first three chapters of the second author's Ph.D. Thesis \cite{V-thesis} under the supervision of the first author.  The authors would like to thank Antun Milas, Drazan Adamovic, and Jinwei Yang for useful comments and guidance in the writing of the thesis from which this paper is derived.  The authors thank the referee for pointing out significant points of clarification in the expository parts of the paper as well as some mathematical improvements.  The first author is the recipient of Simons Foundation Collaboration Grant 282095, and greatly appreciates their support.

\section{Abelian intertwining algebras}

We first recall the notion of abelian intertwining algebra given by Dong and Lepowsky in \cite{DL1}, \cite{DL}. Let $G$ be an abelian group equipped with two functions,
\begin{align*}
F:G\times G\times G\rightarrow \C^\times, \quad \quad \Omega:G\times G\rightarrow \C^\times
\end{align*}
such that $(F,\Omega)$ is a normalized abelian $3$-cocycle for $G$ with coefficients in $\C^\times$ (\cite{MacLane}), that is $F$ and $\Omega$ satisfy
\begin{eqnarray} \label{norm conditions}
F(g_1,g_2,0)=F(g_1,0,g_3)=F(0,g_2,g_3) &=& 1\\
\Omega(g_1,0)=\Omega(0,g_2)&=&1, 
\end{eqnarray}
\begin{multline}
F(g_1,g_2,g_3)F(g_1,g_2,g_3+g_4^{-1})F(g_1,g_2+g_3,g_4)  \\  
 \cdot F(g_1+g_2,g_3,g_4)F(g_2,g_3,g_4)=1, 
\end{multline}
\begin{multline}
F(g_1,g_2,g_3)^{-1}\Omega(g_1,g_2+g_3)F(g_2,g_3,g_1)^{-1} \\
=\Omega(g_1,g_2)F(g_2,g_1,g_3)^{-1}\Omega(g_1,g_3), 
\end{multline}
\begin{multline}\label{norm conditions-last}
F(g_1,g_2,g_3)\Omega(g_1+g_2,g_3)F(g_3,g_1,g_2) =\\
\Omega(g_2,g_3)F(g_1,g_3,g_2)\Omega(g_1,g_3). 
\end{multline}
Denote by
\begin{align}
q:G &\longrightarrow \C^\times  \\
g &\mapsto q(g) =\Omega(g,g) \nonumber
\end{align}  
 the corresponding quadratic form, and $e^{2\pi i b(\cdot,\cdot)}$ its associated (symmetric) bilinear form, where
 \begin{align}
b:G\times G\longrightarrow \C/Z\end{align}
is given by 
\begin{align}
e^{2\pi i b(g_1,g_2)}= q(g_1+g_2)q(g_1)^{-1}q(g_2)^{-1}(=\Omega(g_1,g_2)\Omega(g_2,g_1))
\end{align} 
for $g_1,g_2\in G$. Define
\begin{align}
B: G\times G\times G& \longrightarrow \C^\times \\
 (g_1,g_2,g_3)&\mapsto F(g_2,g_1,g_3)^{-1}\Omega(g_1,g_2)F(g_1,g_2,g_3).  \nonumber
\end{align}
Let $T$ be a positive integer. Assume that $b$ is restricted to the values
\begin{align}\label{T} 
b:G\times G\rightarrow \left(\frac{1}{T}\Z\right)/ \Z
\end{align} 

\begin{defn} An {\em abelian intertwining algebra of level $T$ associated with $G$, $F$, and $\Omega$} is a vector space $V$ with a $(1/T)\Z$-$G$ bigradation 
\begin{eqnarray}
\tilde{V}=\coprod_{g\in G}\tilde{V}^{g}= \coprod_{n\in \frac{1}{T}\Z }\tilde{V}_{(n)}
\end{eqnarray}	
equipped with a linear map
\begin{eqnarray}
Y: \tilde{V} &\longrightarrow& (\End \tilde{V})[[x^{\frac{1}{T}},x^{\frac{-1}{T}}]]\\ 
u &\mapsto& Y(u,x)=\sum_{n\in \frac{1}{T}\Z }u_{n}x^{-n-1} \nonumber
\end{eqnarray}
called the {\em generalized vertex operator map} and with a distinguished vector ${\bf 1}\in \tilde{V}^{0}\cap \tilde{V}_{(0)}$,
called the {\em vacuum vector},
satisfying the following conditions for
$g,h\in G,\; u,v\in \tilde{V}$ and $l\in \frac{1}{T}\Z$:
\begin{eqnarray}
& &u_{l}\tilde{V}^{h}\subset \tilde{V}^{g+h}\;\;\;\mbox{ if }u\in V^{g};\\
& &u_{l}v=0\;\;\;\mbox{ if $l$ is sufficiently large};
\label{etruncation}\\
& &Y({\bf 1},x)=1;\label{evacuum}\\
& &Y(v,x){\bf 1}\in \tilde{V}[[x]]\;\;\mbox{ and }
v_{-1}{\bf 1}\;(=\lim_{x\rightarrow 0}Y(v,x){\bf 1})=v;\\
& &Y(u,x)|_{\tilde{V}^{h}}=\sum_{n\equiv b(g,h)\mod \Z}u_{n}x^{-n-1}
\;\;\;\mbox{ if }u\in \tilde{V}^{g}
\end{eqnarray}
\begin{eqnarray}\label{egjacobi0}
& & \ \ \ x_{0}^{-1}\left(\frac{x_{1}-x_{2}}{x_{0}}\right)^{b(g,h)}
\delta\left(\frac{x_{1}-x_{2}}{x_{0}}\right)Y(u,x_{1})Y(v,x_{2})w\\
& &-B(g_1,g_2,g_3)x_{0}^{-1}\left(\frac{x_{2}-x_{1}}{x_{0}}\right)^{b(g,h)}
\delta\left(\frac{x_{2}-x_{1}}{-x_{0}}\right)Y(v,x_{2})Y(u,x_{1})w\nonumber\\
&=&F(g_1,g_2,g_3)x_{2}^{-1}\delta\left(\frac{x_{1}-x_{0}}{x_{2}}\right)Y(Y(u,x_{0})v,x_{2})
\left(\frac{x_{1}-x_{0}}{x_{2}}\right)^{-b(g,g')}w \nonumber
\end{eqnarray}
(the {\em generalized Jacobi identity}) if $u\in \tilde{V}^{g}$, $v\in \tilde{V}^{h}$, $w \in \tilde{V}^{g'}$, and $\delta(x)=\sum_{n\in\Z}x^n$.
In addition there exists an $\omega\in \tilde{V}^0\cap \tilde{V}_{(2)}$ such that $Y(\omega,x)= \sum_{n\in\Z} L(n)x^{-n-2}$, where the $L(n)$ give a representation of the Virasoro algebra with central charge $\textbf{c}\in \C$, and $L(0)$ acts semisimply on $\tilde{V}$ giving the $\frac{1}{T}\mathbb{Z}$ grading, i.e.,  $\tilde{V}_{(r)}= \{v\in \tilde{V} \; | \; L(0)v=r v  \}$, and  $Y(L(-1)v,x)= \frac{d}{dx} Y(v,x)$ for any $v\in \tilde{V}$.  
\end{defn}

This structure is denoted 
\begin{align*}
(\tilde{V},Y,\textbf{1},\omega,T,G,F(\cdot,\cdot,\cdot), \Omega(\cdot,\cdot)) 
\end{align*}
or briefly $\tilde{V}$.  If $F$ is the trivial cocycle, $\tilde{V}$ is called a {\it generalized vertex operator algebra}.  If $T=1$ and $G=0$, we have the notion of a {\it vertex operator algebra}.

\begin{defn} \label{canonical VOA inside GVOA}	
	From the normalization conditions (\ref{norm conditions})--(\ref{norm conditions-last}) combined with (\ref{egjacobi0}), the space  $$V=\tilde{V}^0 \cap  \coprod_{n\in\frac{1}{T}\Z}\tilde{V}_{(n)}$$ has the structure of a vertex algebra and we will call it {\em the vertex algebra associated with $\tilde{V}$}.  In the case $G$ has even order, suppose there exists $k\in G$ with $2k=0$ and $g(k,k)=1 \mbox{ mod } \Z/2\Z$. Then the space  
	\begin{align}
	V&= (\tilde{V}^0 \cap  \coprod_{n\in \frac{1}{T}\Z}\tilde{V}_{(n)})\oplus ( \tilde{V}^{k}\cap  \coprod_{n\in \frac{1}{T}\Z}\tilde{V}_{(n)}) \label{vosa}  \\
	&:=  V^{(0)}\oplus V^{(1)},\mbox{ where } \{(0),(1)\}=\Z/2\Z\subset G 
	\end{align}  is {\em a vertex operator superalgebra associated to $\tilde{V}$}. Such a vertex superalgebra is uniquely determined if and only if $k$ is uniquely determined by the conditions above.  If in addition we assume that the following two finiteness conditions hold:  

	1. $\tilde{V}_{(r)}\cap V^0=0$ for $r\in \Z $ sufficiently negative; 

	2. $\tilde{V}_{(r)}\cap V^0$ is finite dimensional for all $r\in \mathbb{\Z} $, \\
	then the space $$V=\tilde{V}^0 \cap  \coprod_{n\in\Z}\tilde{V}_{(n)}$$ is the {\em vertex operator algebra associated with $\tilde{V}$}, and in the case $\tilde{V}_{(r)}\cap V^k$ has the same finiteness conditions for $r\in \frac{1}{2}\Z $, the space  
	\begin{align}
	V&= (\tilde{V}^0 \cap  \coprod_{n\in \Z}\tilde{V}_{(n)})\oplus ( \tilde{V}^{k}\cap  \coprod_{n\in \frac{1}{2}\Z}\tilde{V}_{(n)}) \label{vosa2}  
	\end{align}  is a {\em vertex operator superalgebra associated to $\tilde{V}$}.
\end{defn} 

From now on we will assume that $\tilde{V}$ is an abelian intertwining algebra such that the finite conditions 1 and 2 in the definition above hold.

Below we define various notions of module for a vertex operator superalgebra, manly to contextualize the introduction and motivation for why screening operators are important and in what settings.

For a vertex operator superalgebra $V$, a weak $V$-module $M$ is a vector space $M$ and a map $Y_M:V\longrightarrow (\mathrm{End} M ) [[x, x^{-1}]]$ such that $u_{l}w=0$ for $l$ sufficiently large, $Y_M(\mathbf{1},x)=id_M$, and a slightly modified Jacobi identity as in (\ref{egjacobi0}) holds, (see for instance \cite{LL} for more details).

An {\it $\mathbb{N}$-gradable weak $V$-module} (also often called an {\it admissible $V$-module}) $M$ for a vertex operator algebra $V$ is a weak $V$-module that is $\N$-gradable, $M = \coprod_{k \in \N} M(k)$, with $v_m M(k) \subset M(k + \mathrm{wt} v - m -1)$ for weight homogeneous $v \in V$, $m \in \Z$ and $k \in \N$, and without loss of generality, we can assume $M(0) \neq 0$.  We say elements of $M(k)$ have {\it degree} $k \in \mathbb{N}$.

An {\it $\mathbb{N}$-gradable generalized weak $V$-module} $M$ is an $\mathbb{N}$-gradable weak $V$-module that admits a decomposition into generalized eigenspaces via the spectrum of $L(0) = \omega_1$ as follows: $M=\coprod_{\lambda \in{\C}}M_\lambda$ where $M_{\lambda}=\{w\in W \, | \, (L(0) - \lambda Id_M)^j w= 0 \ \mbox{for some $j \in \mathbb{Z}_+$}\}$, and in addition, $M_{n +\lambda}=0$ for fixed $\lambda$ and for all sufficiently small integers $n$. We say elements of $M_\lambda$ have {\it weight} $\lambda \in \mathbb{C}$.

A {\it generalized $V$-module} $M$ is an $\mathbb{N}$-gradable weak $V$-module where $\dim W_{\lambda}$ is finite for each $\lambda \in \mathbb{C}$.   

An {\it (ordinary) $V$-module} is an $\mathbb{N}$-gradable generalized weak $V$-module such that  the generalized eigenspaces $M_{\lambda}$ are in fact eigenspaces, i.e. $M_{\lambda}=\{w\in M \, | \, L(0) w=\lambda w\}$.

We often omit the term ``weak" when referring to $\mathbb{N}$-gradable weak and $\mathbb{N}$-gradable generalized weak $V$-modules.   

A  $V$-module $M$ is {\it irreducible} or {\it simple} if $M$ has no nontrivial proper submodules. 
The term {\it logarithmic} is also often used in the literature to refer to $\mathbb{N}$-gradable weak generalized modules  or generalized modules, particularly when these modules are non simple indecomposable.  

A vertex operator algebra $V$ is said to be {\it rational} if every $\N$-gradable module decomposes into a direct some of irreducible modules.  

We denote by $C_2(V)$ the vector space spanned by elements of the form $u_{-2}v$ for $u,v\in V$. We say $V$ is {\it $C_2$-cofinite} if the vector space $V/C_2(V)$ is finite dimensional.

\section{The notion of a screening operator}

Let $\tilde{V}$ be an abelian intertwining algebra satisfying the finiteness conditions 1 and 2 in Definition \ref{canonical VOA inside GVOA}, and let $V\subset \tilde{V}$ be a vertex operator superalgebra as in Definition \ref{canonical VOA inside GVOA}.  Suppose that for $g,g',h\in G$, we have that $(g,g'),(g,h)\in \Z\mbox{ mod }\Z/2\Z$.  Then extracting the coefficient $x_0^{-1}x_1^{-1}$ from the Jacobi identity (\ref{egjacobi0}), the endomorphism $u_0$ for $u \in \tilde{V}^g$ acts as a derivation in the sense that: 
\[ u_0 Y(v,x_2)w= Y(u_0 v,x_2)w + (-1)^{(g,h)}Y(v,x_2)u_0 w,\]
for $v\in \tilde{V}^h$, and $w\in \tilde{V}^{g'}$. In particular, this holds whenever $u,v \in V$ as in (\ref{vosa}).  However, unless $T=1$ in $(\ref{T})$, $a_0$ does not act as a derivation in this sense on all of $\tilde{V}$. This motivates the next definition:

\begin{defn} \label{screening operator definition}
	Let $\tilde{V}$ be an abelian intertwining algebra with vertex operator subsuperalgebra $V$.	We define the space of {\em graded derivations associated to $V$ and $\tilde{V}$} to be ${\rm GrDer}(V,\tilde{V})=\coprod_{g\in G} {\rm GrDer}_g$, where ${\rm GrDer}_g$, for $g \in G$, are the linear maps $f:V\rightarrow \tilde{V}^g\oplus \tilde{V}^{g+k}$ satisfying:

(i) $f(u_n v)= (f(u))_n v+ (-1)^{(g,h)} u_n (f(v))$;  

(ii) $[f,L(-2)]=0$ and $f(\mathbf{1}) =0$;\\
for  $u\in V^{(h)}$, $v\in V$, $h\in \Z/2\Z\subset G $, and $n\in\Z$.

If $f\in {\rm GrDer}(V, \tilde{V})$ such that $f$ is nilpotent when projected onto $V$,  then we call $f$ a {\em screening operator on $V$. }
\end{defn}

\begin{lem} \label{lem 3.2} Let $f\in{\rm GrDer}(V,\tilde{V})$.  Then 

(i) $\omega\in \mbox{ker} \, f$ and $[f,L(n)]=0$; 

(ii) $\mbox{ker} \, f$ is a vertex operator subsuperalgebra of $V$.
\end{lem}  

\begin{proof}
To prove property $(i)$, we note that since $f$ satisfies condition $(ii)$ in Definition \ref{screening operator definition}, we have $f\omega= fL(-2)\mathbf{1}= L(-2)f\mathbf{1}=0$.  Furthermore, this and property $(i)$ in Definition \ref{screening operator definition}, imply $[f,L(n)]= (f\omega)_{n+1}=0$. 

To prove $(ii)$, we use  $(i)$ and the fact that for $f\in{\rm GrDer}_g$, $u\in V^{(i)}$, for $i = 0$ or $1$, and $u,v \in \mbox{ker} \, f$, we have $f(u_n v)= (fu)_n v+ (-1)^{(g,(i))}u_n f(v)=0$, so that $u_n v\in \mbox{ker} \, f$. Then $\mbox{ker} \, f$ is a vertex operator subsuperalgebra of $V$. 
\end{proof}


We will now discuss the general technique for finding certain special kinds of screening operators that come from an abelian intertwining algebra.  Let $v\in \tilde{V}$ be a weight-homogeneous vector.  Then from the generalized Jacobi identity (\ref{egjacobi0}),
\begin{eqnarray*}  
[L(n),v_m] &=&  (L(-1)v)_{n+m+1}+  (n+1)(\mbox{wt} \, v) v_{n+m}  \\ 
& & \quad + \sum_{i\geq 2} \binom{n+1}{i}  (L(i-1)v)_{n+m+1-i}  \\
&= & (-m+ (n+1)(\mbox{wt} \, v-1)) v_{n+m} \\
& & \quad + \sum_{i\geq 2} \binom{n+1}{i}  (L(i-1)v)_{n+m+1-i}. \\ 
\end{eqnarray*} 

If $v$ is a primary vector, that is a singular vector for the Virasoro algebra, we have that 
 \begin{align*}  
 [L(-2),v_m] &=  (-m -(\mbox{wt} \, v-1)) v_{m-2}
 \end{align*} 
 and in particular, 
  \begin{align} \label{primary commutator}
  [L(-2),v_{-{\tiny\mbox{wt}} \, v+1}] &= 0.
  \end{align}
Applying equation \ref{primary commutator} to $\mathbf{1}$ gives that condition $(ii)$ in Definition \ref{screening operator definition} is satisfied if and only if wt $v\leq 1$. Furthermore, condition $(i)$ in Definition \ref{screening operator definition} is satisfied when wt $v= 1$. In this case, $v_0$ is a screening operator and $\mbox{ker} \, v_0$ is a vertex operator subsuperalgebra of $V$. 

Although screening operators of the form $v_0$ for $v\in \tilde{V}$ a weight one primary vector are the most tractable class of screening operators, $\mbox{ker} \, v_0$ is still difficult to study.  We now introduce some additional structure that has proved to be fruitful for studying these subalgebras. We will say that $v\in\tilde{V}$ has order $r\in \Z_+$ if there exists a smallest $r> 0$ such that $v_{n_1} \cdots  v_{n_r}\mathbf{1}\in V$ for all $n_j\in \frac{1}{T}\Z$, $j=1,\dots,r$, and $n_1+\cdots + n_r \in\Z$. If such an $r$ exists, we will say $p=nr$ for $n\in \Z_+$ is a {\it period} of $v$.  

\begin{defn} \label{screening pair definition} 
	We call two screening operators arising as the zero modes of weight $1$ primary vectors $v_1,v_2\in \tilde{V}$ a {\em screening pair}, denoted $((v_1)_0,(v_2)_0)$, of type $(p,q)$ for $p,q\in \Z_+$ if the following hold: 

	$(i)$ $[(v_1)_0,(v_2)_0]=0$; 

	$(ii)$  $v_1$ has period $q$ and $v_2$ has period $p$, and $p,q$ are the smallest positive integers such that there exists $n_1,\dots n_q,m_1, m_2\dots m_p\in \frac{1}{T}\Z$ and $m\in \Z$ with $n_1+\cdots + n_q \in\Z$ and $m_1+\cdots+ m_p\in\Z$ such that 
\[((v_1)_{n_1} \cdots  (v_1)_{n_q}\mathbf{1})_m (v_2)_{m_1} \cdots  (v_2)_{m_p}\mathbf{1}=r\mathbf{1}.\]
	for some $r\in \C^\times.$ 
\end{defn}
\begin{rema}
It follows that if $((v_1)_0,(v_2)_0)$ is a screening pair of type (p,q), then $((v_2)_0,(v_1)_0)$ is a screening pair of type $(q,p)$.
\end{rema}

Here we give some motivation as to why screening pairs should be defined in this way and why they are important.  Condition $(i)$ in Definition \ref{screening pair definition} implies that if $u\in \mbox{ker} \, (v_2)_0$, then $(v_1)_0^{jq} u\in\mbox{ker} \, (v_2)_0$ for any $j\in\Z_+$. Moreover, if $w$ is a primary vector, then from Lemma \ref{lem 3.2} so is $((v_1)_0)^j w$ for any $j\in\Z_+$.  Condition $(ii)$ of Definition \ref{screening pair definition} then implies that  
applying $((v_1)_0)^q$ to ``enough" elements in $\mbox{ker} \, (v_2)_0$ allows one to find a decomposition of $\mbox{ker} \, (v_1)_0$ as a $\mbox{ker} \, (v_1)_0\cap \mbox{ker} \, (v_2)_0$-module and obtain a nice strongly generating set, if such a set exists.

In \cite{A-2003},\cite{am1}--\cite{am4}, screening pairs as defined above were the essential ingredients used to construct the triplet vertex operators algebra $\mathcal{W}(p,q)$ of central charge $c_{p,q}$ from the vertex operator algebra associated with a rank one even lattice $L=\sqrt{2pq}\Z$, and were crucial ingredients in proving for instance that $\mathcal{W}(p,1)$ is simple and $C_2$-cofinite. Following the notation of \cite{AdM}, \cite{am1}, in the case that $q=1$, we will call $(v_2)_0=\tilde{Q}$ the {\it short screening operator} and $(v_1)_0=Q$ the {\it long screening operator}. In \cite{BV-W-algebras} (cf. \cite{V-thesis}), we study the kernel of long screening operator in a screening pair and the intersections of the kernel of multiple screening operators under certain conditions.

A  {\it weak homomorphism} of two generalized vertex operator algebras $(V_1,\omega_1)$ and  $(V_2,\omega_2)$ is a linear map $f:V_1\rightarrow V_2$ satisfying $f(u_n v) = f(u)_n f(v)$ for $u,v \in V_1$, $f(\mathbf{1}_1)=\mathbf{1}_2$. A weak homomorphism is a {\it homomorphism} if in addition $f(\omega_1)=\omega_2$.  We will denote by $\mbox{wAut} (\tilde{V})$ the group of weak generalized vertex operator algebra automorphisms of $\tilde{V}$.

The next remark easily follows from the definitions:

\begin{rema} \label{lem for isomorphism} 
	Let $V$ be a vertex superalgebra. Let $\omega_1$ and $\omega_2$ be elements of $V$ such that $(V,\omega_1)$ and $(V,\omega_2)$ are vertex operator superalgebras. Let $(V_1,\omega_1)$, $(V_2,\omega_2)$ be vertex operator subsuperalgebras of $(V,\omega_1)$, $(V,\omega_2)$, respectively. Let $g$ be a weak homomorphism of $V$ that sends $V_1$ into $V_2$ and satisfies
	$g(\omega_1)=\omega_2$. Then $g$ induces a vertex operator superalgebra homomorphism from $V_1$ into $V_2$. 
\end{rema} 

 The next proposition gives an upper bound on non-isomorphic equivalence classes of vertex operator subsuperalgebras contained in a fixed vertex operator superalgebra.

\begin{prop}\label{prop for isomorphism} Let $V$, $\omega_1$, $\omega_2$, $g$ be as in Remark \ref{lem for isomorphism} with $V\subset \tilde{V}$, for $\tilde{V}$ an abelian intertwining algebra and assume $g$ can be extended to a weak homomorphism on $\tilde{V}$. Let $v^1_0,v^2_0\in {\rm GrDer}(V,\tilde{V})$ with $v^1,v^2\in \tilde{V}$ such that $\omega_1\in \mbox{ker} \, v^1_0$ and $\omega_2\in \mbox{ker} \, v^2_0$, and such that $g(v^1)=v^2$. Then $g$ induces a homomorphism of vertex operator algebras: 
	$$ g|_{V}: \mbox{ker} \, v^1_0|_V\longrightarrow \mbox{ker} \, v^2_0|_V. $$
	If $g\in \mbox{wAut}(V)$, then $g|_{V}$ is a vertex operator superalgebra isomorphism.
\end{prop}  

\begin{proof}
Note that if $v\in \mbox{ker} \, (v^1_0)$, then $$0=g(v^1_0 v)= v^2_0 gv,$$ so that $gv\in \mbox{ker} \, (v^2_0)$. Let $v\in   \mbox{ker} \, v^2_0$. If $g$ is a bijection on $V$, we have that $v=gu$ for some $u\in V$, and moreover, $u\in \mbox{ker} \, (v^1_0)$, since 
$$0= v^2_0v= v^2_0g u= g(v^1_0 u).$$ Hence $g$ is an isomorphism. 
\end{proof} 

\begin{cor}
	Let $g$, $\omega_1$, $\omega_2$, and  $\tilde{V}$ satisfy the assumptions in Proposition \ref{prop for isomorphism}. Let $v_1,\dots, v_n$  be a collection of weight $1$ primary vectors for $(\tilde{V},\omega_1)$, and $v'_1,\dots v'_n$  be a collection of weight $1$ primary vectors for $(\tilde{V},\omega_2)$ such that $g v_i=v_i'$ for $1\leq i\leq n$. Then $g$ induces an isomorphism $g|_{V}$ of vertex operator superalgebras: 
	$$ g|_{V}: \cap_{i=1}^n \mbox{ker} \, ((v_i)_{0})|_V\longrightarrow \cap_{i=1}^n\mbox{ker} \, ((v'_i)_{0}|_V. $$
\end{cor}

\section{Screening pairs associated to lattice vertex operator superalgebras}

This section is concerned with the existence and classification of the possible types of screening pairs that arise in a lattice vertex operator superalgebra, finding that there are 4 general types--- Theorem 4.3. We will then go on in Section 5 to classify one of these types of screening pairs for all rank 2 lattices, and for all lattices that are generated by lattice elements that give rise to a screening pair.

We will first review some facts about positive definite integral lattices, and give the basic construction for the vertex operator superalgebra associated to such a lattice, following for instance \cite{LL}, \cite{BV1}, \cite{Xu}. For more information on the theory of lattices, see for instance \cite{Ebe}. 

Let $L$ be a positive definite integral lattice.  That is, $L$ is a finitely generated free abelian group with $\Z$-valued positive definite symmetric bilinear form $\langle\cdot,\cdot \rangle$, with natural $\Z_2$-grading $L=L^{(0)}\cup L^{(1)}$, where $L^{(i)}= \{\alpha\in L\,|\, \langle\alpha,\alpha\rangle=i+2\Z\}$. It is clear that one can embed $L\subset \mathbb{R}^d$, where $d$ is the rank of the lattice and $\langle \cdot , \cdot\rangle$ is the usual dot product.  If we choose an ordered $\Z$-basis $B=\{\alpha_i\}_{i=1}^d$ in $\mathbb{R}^d$, one can associate a Gram matrix $G=(\langle \alpha_i,\alpha_j\rangle)_{i,j=1}^d= B^T B$ with the positive definite assumption implying $\mbox{Det}(G)>0$.  Two lattices $L$ and $L'$ are isomorphic if there exists an isometry $\nu:L\rightarrow L'$; that is, $\nu$ is a group isomorphism that further satisfies $\langle \nu \alpha,\nu\beta \rangle'=\langle  \alpha,\beta \rangle$.   
Similarly, a lattice with ordered $\Z$-bases $B$ and $B'$ and Gram matrices $G$ and $G'$ are isomorphic if there exists a matrix $A\in \mbox{Gl}_d(\Z)$ with $G'=AGA^{T}$. 
We define the rational lattice dual to $L$ to be:
$$L^\circ=\{\beta\in \Q\otimes L| \, \langle \alpha,\beta\rangle\in \Z \mbox{ for all } \alpha\in L\}$$ and also the {\em extended dual space}, 
\begin{align} \label{extended dual}
\bar{L}^\circ=\{\beta\in \Q\otimes L| \, \langle \alpha,\beta\rangle\in (1-\frac{i}{2})\Z \mbox{ for all } \alpha\in L^{(i)}\}.
\end{align} 
Notice that $L^\circ/L$ is a finite abelian group, and $|L^\circ/L|=\mbox{Det } G$.

$L$ is called unimodular if $\mbox{Det } G=1$, or equivalently $L=L^\circ$.  For lattices, $L_1$ and $L_2$, let $ L_1\oplus L_2$ denote the orthogonal direct sum of $L_1$ and $L_2$. We call $L$ indecomposable if it can not be written as an orthogonal direct sum of two lattices of smaller rank. Given a rank 1 lattice with generator $\alpha$ satisfying $\langle\alpha,\alpha\rangle=p$, we will refer to the corresponding lattice as $\Z\alpha$ or $\sqrt{p}\Z$.  

We will need some basic facts arising in this setting.  Firstly from basic linear algebra, we have:

\begin{lem} (\cite{N}, p.\ 13) \label{M}
	Let $d_1,\dots, d_n\in \Z$, not all zero, and let $\delta_n= \mbox{ gcd }(d_1,\dots, d_n)$. Then there is an $n \times n$ matrix $M_n$ over $\mathbb{Z}$ with first column $[d_1,\dots,d_n]$ and determinant $\delta_n$. 
\end{lem}

In addition, we have the following:

\begin{lem} \label{first lem}
	Let $L$ be a rank $d$ positive definite integral lattice. 

	(i)  For a fixed $\alpha\in L$, if there exists a $\beta\in L^\circ$ such that $\langle \alpha,\beta\rangle =1$, then there exists a $\Z$-basis $B$ for $L$ containing $\alpha$. 

	(ii)  If $\langle \alpha,L\rangle\subset \langle \alpha,\alpha\rangle \Z$, then there exists a rank $d-1$ lattice $L'$ with $L$ having orthogonal decomposition
\[L= \Z\alpha\oplus L'.\]

	(iii)  $\langle \alpha, L^\circ\rangle \subset n\Z$ if and only if $\alpha\in nL.$
\end{lem}

\begin{proof} For {\it (i)}, pick an ordered $\Z$-basis $B=\{\alpha_1,\dots,\alpha_d\}$ for $L$. Let $\alpha\in L$ with $\alpha=\sum_{i=1}^d a_i \alpha_i$ for $\alpha_i\in B$ and $a_i\in \Z$. If there exists a $\beta\in L^\circ$ such that $\langle \alpha,\beta\rangle =1$, then writing $\beta=\sum_{i=1}^d b_i \alpha_i^\circ$ for $b_i\in \Z$ and where $\{\alpha_1^\circ,\dots,\alpha_d^\circ \}$ denotes the dual basis to $B$, we have that $\sum_{i=1}^d a_i b_i=1$, implying  gcd$(a_1,\dots, a_d)=1$.  Conversely, if gcd$(a_1,\dots, a_d)=1$, then there exists $b_1,\dots b_d\in\Z$ such that $\sum_{i=1}^d a_i b_i=1$, and setting $\sum_{i=1}^d b_i \alpha_i^\circ$ gives a $\beta\in L^\circ$ such that $\langle \alpha,\beta\rangle =1$. Since gcd$(a_1,\dots, a_d)=1$,  by Lemma \ref{M} there is a matrix $P\in \mbox{GL}_d(\Z)$ with $[a_1,\dots,a_d]$ as one of its columns. 
Then $P^{-1}\in\mbox{GL}_d(\Z)$ gives a change of $\Z$-basis from the $\Z$-basis $B$ to a new $\Z$-basis $B'$ with $\alpha\in B'$.  

Now assume $\langle \alpha,L\rangle\subset \langle \alpha,\alpha\rangle \Z$ as in {\it (ii)}.  Since $\alpha/\langle\alpha,\alpha\rangle \in L^\circ$ and $\langle\alpha, \alpha/\langle\alpha,\alpha\rangle\rangle=1$, it follows from {\it(i)} that there exists a $\Z$-basis $B$ that contains $\alpha$.  Construct a new $\Z$-basis by replacing $\alpha_i\in B$ with $\hat{\alpha_i}=\alpha_i- (\langle\alpha,\alpha_i\rangle/\langle\alpha,\alpha\rangle)\alpha$ when $\alpha_i\neq \alpha$. Then $\langle\alpha,\hat{\beta}\rangle=0$ and $\hat{B}$ give a $\Z$-basis for $L'$.
 
For {\it(iii)}, suppose $\alpha\in n L$. Then using the $\Z$-basis $B$ in the proof of part {\it(i)} above, we can again write $\alpha=\sum_{i=1}^d a_i \alpha_i$ with each $a_i\in n \Z$ and so $\langle \alpha, \alpha_i^\circ\rangle
\in n\Z$, whence  $\langle \alpha, L^\circ\rangle \subset n\Z$.  Conversely, suppose  $\langle \alpha, L^\circ\rangle \subset n\Z$. Then writing $\alpha=\sum_{i=1}^d a_i \alpha_i$ we have $\langle \alpha,\alpha_i^\circ\rangle\in n\Z$, so that $a_i\in n\Z$ for each $i$. Hence $\alpha\in nL$.   
\end{proof}

We now recall the construction of the bosonic vertex operator algebra $M(1)$ and the lattice vertex operator superalgebra $V_L$ (see e.g. \cite{LL}). Starting with a positive definite integral lattice $L$, let $\h= L\otimes_{\Z} \C$ and let $\hat{{\h}}={\C}[t,t^{-1}]\otimes {\h} \oplus {\C}c$ be the affinization of
${\h}.$
Set
$\hat{{\h}}^{+}=t{\C}[t]\otimes
{\h};\;\;\hat{{\h}}^{-}=t^{-1}{\C}[t^{-1}]\otimes {\h}.$
Then $\hat{{\h}}^{+}$ and $\hat{{\h}}^{-}$ are abelian subalgebras
of $\hat{{\h}}$. Let $U(\hat{{\h}}^{-})=S(\hat{{\h}}^{-})$ be the
universal enveloping algebra of $\hat{{\h}}^{-}$. Let ${\l} \in\C$. Consider the induced $\hat{{\h}}$-module
\begin{eqnarray*}
	M(1,{\l})=U(\hat{{\h}})\otimes _{U({\C}[t]\otimes {\h}\oplus
		{\C}k)}{\C}_{\l}\simeq
	S(\hat{{\h}}^{-})\;\;\mbox{(linearly)},\end{eqnarray*} where
$t{\C}[t]\otimes {\h}$ acts trivially on ${\C}_{\l}$,
${\h}$ acts as $\la h, {\l} \ra$ for $h \in {\h}$
and $k$ acts on ${\C}_{\l}$ as multiplication by 1. We shall write
$M(1)$ for $M(1,0)$.
For $h\in {\h}$ and $n \in {\Z}$ write $h(n) =  t^{n} \otimes h$. Set
$
h(z)=\sum _{n\in {\Z}}h(n)z^{-n-1}.
$

We can choose a Virasoro element of the form 
\begin{equation} \label{virasora grading} \omega_\gamma \ = \ \sum_{i=1}^d \frac{h_i(-1)^2}{2}{\bf 1}+\gamma(-2){\bf 1} \  = \ \omega_{st}+ \gamma(-2){\bf 1},
\end{equation}
where $\{h_i\}_{i=1}^d$ is an orthonormal basis for $\h$ and $\gamma\in \h$, and $\omega_{st}$ denotes that this is the standard Virasoro element often used in the literature.
Then $(M(1),\omega_\gamma)$ is a vertex operator algebra which is generated by the fields
$h(z)$ for $h \in {\h}$, and $M(1,{\l})$, for $\l \in \C$, are
irreducible modules for $M(1)$.

For $L$ a positive definite integral lattice, we shall denote by $V_L$ the vector space, \[V_{L}=M(1)\otimes\mathbb{C}[L],\]
where $\mathbb{C}[L]$ is the group algebra of $L$ with generators $e^{\alpha}$ for $\alpha\in L$. Then $V_L=(V_L,Y,\omega_{\gamma},{\bf 1})$ for $\gamma$ such that $\gamma\in \bar{L}^\circ\subset \h$ can be given the structure
of a vertex operator algebra,  with vacuum vector given by
\[{\bf 1}=1 \otimes e^0\]
and necessarily $\gamma\in \bar{L}^\circ$ in (\ref{virasora grading}) with $\bar{L}^\circ$ defined in (\ref{extended dual}).  The grading is given by 
\begin{align} \label{grading}
(\omega_\gamma)_1 (m_1\otimes e^{\alpha})=L_\gamma(0) (m_1\otimes e^{\alpha})= r+\langle \alpha,\alpha\rangle/2- \langle\gamma,\alpha\rangle
\end{align} 
for $m_1\in M(1)_r$ and $\alpha\in L^\circ$. For more information on the full construction of $V_L$, see for instance (\cite{D},\cite{DL},\cite{FLM},\cite{LL})).
Moreover,
\[V_{L^\circ}=M(1) \otimes\mathbb{C}[L^\circ ],\]
has the structure of an abelian intertwining algebra of level $T$ with $G=L^\circ/L$ and $(\cdot,\cdot)$ is naturally determined by setting $(\cdot,\cdot)=\langle\cdot,\cdot \rangle \mbox{ mod }2\Z$. Furthermore, $V_L$ is a vertex operator subsuperalgebra of $V_{L^\circ}$. (See \cite{DL}).

It is known that the vertex operator superalgebra $V_L$ is rational, in
the sense that it has finitely many irreducible $V_L$-modules and
that every $V_L$-module is completely reducible (see \cite{LL} e.g.). In fact,
\[V_{L+\lambda}=M(1) \otimes e^{\lambda} \mathbb{C}[L], \ \ \ \lambda\in L^\circ/L,\]
are, up
to equivalence, all irreducible $V_L$-modules. For $\lambda=0$ we
recover the vertex algebra $V_L$, while $V_{L^\circ}$ has a
generalized vertex operator algebra structure. The Virasoro
algebra acts on $V_{L+\lambda}$ with the central charge
$$\textbf{c}=d-12 \langle \gamma,\gamma\rangle .$$
The Virasoro generators will be denoted by $L_\gamma(n)$, $n \in \mathbb{Z}$.

Define the Schur polynomials $S_{r}(x_{1},x_{2},\dots)$
in variables $x_{1},x_{2},\dots$ by the following equation:
\begin{eqnarray}\label{eschurd}
\exp \left(\sum_{n= 1}^{\infty}\frac{x_{n}}{n}y^{n}\right)
=\sum_{r=0}^{\infty}S_{r}(x_1,x_2,\dots)y^{r}.
\end{eqnarray}
For any monomial $x_{1}^{n_{1}}x_{2}^{n_{2}}\cdots x_{r}^{n_{r}}$
we have an element $$h(-1)^{n_{1}}h(-2)^{n_{2}}\cdots
h(-r)^{n_{r}}{\1} $$ in $M(1)$   for $h\in{\h}.$
Then for any polynomial $f(x_{1},x_{2}, \dots)$, it follows that $f(h(-1),
h(-2),\dots){\1}$ is a well-defined element in $M(1)$. In
particular, $S_r (h)=S_{r}(h(-1),h(-2),\dots){\1} \in M(1)$ for $r \in
{\N}$. Extend $S_r (h)$ to $r\in\Q$ by setting $S_r (h)=0$ for $r\notin \N$. Fix $T\in \Z_+$ such that $T\cdot \langle\cdot ,\cdot \rangle$ is $\Z$-valued on $L^\circ$. Then for $\alpha,\delta\in L^\circ$, we have that 
\bea\label{eab} && e ^{\alpha}
_{i-1} e ^{\delta} =S_{-\langle\alpha,\delta\rangle-i}(\alpha) e ^{\alpha + \delta}, \eea
for $i\in\frac{1}{T}\Z$.

Next we study screening pairs inside $V_{L^\circ}$ for Virasoro elements of the form $\omega_\gamma$ for various $\gamma\in \bar{L}^\circ$.  We will focus on screening operators coming from weight $1$ primary vectors of the form $m\otimes e^{\alpha }$ for $m\in M(1)$ and $\alpha\in L^\circ$, since these are the most tractable. The following theorem describes precisely what such screening pairs must look like:

\begin{thm} \label{thm for screening pairs} 
Every screening pair with screening operators comprised of weight $1$ primary vectors of the form $m\otimes e^{\alpha }$ for $m\in M(1)$ and $\alpha\in L^\circ$ is of one of the four forms:\\


	(i) $(e^{-\alpha/p}_0,e^{\alpha/p'}_0)$ with $\alpha\in L^{(0)}$, satisfying $\langle \alpha,\gamma\rangle = p-p'$ and $\langle \alpha,\alpha\rangle = 2pp'$ for positive integers $p,p'$, and $\alpha/p,\alpha/p'\in L^\circ$. \\

	(ii) $(e^{-\alpha/p}_0,(\beta(-1)e^{(-p'+p)\alpha/pp'})_0)$ with $\alpha\in L^{(0)}$, $\beta\in \h$ not in the span of $\alpha$, $\langle \alpha,\gamma\rangle = p-p'$ and $\langle \alpha,\alpha\rangle = 2pp'$, where $\alpha/p,$ $(-p'+p)\alpha/pp' \in L^\circ$, and satisfying $\langle\beta,(-p'+p)\alpha/pp'-2\gamma\rangle=0$, and $p>p'$ for positive integers $p,p'$, and $p\neq 2p'$. \\

	(iii) $((\beta(-1)e^{-\alpha/p})_0,(e^{(-p'+r)\alpha/2p p' })_0)$ with $\alpha\in L^{(0)}$, $\beta\in \h$ not in the span of $\alpha$, $\langle \alpha,\gamma\rangle =-p' $ and $\langle \alpha,\alpha\rangle = 2pp'$ with $p=(r^2-p'^2)/4p'\in\Z_+$ for suitable $r>p'$ such that $r\neq 3p'$, and satisfying $\alpha/p, (-p'+r)\alpha/2p p' \in L^\circ$, and $\langle\beta,\alpha/p+2\gamma\rangle=0$.\\

	(iv)  $((m_1\otimes e^{-\alpha/p})_0, e^{s \alpha}_0)$ and $( e^{-\alpha/p}_0, (m_2\otimes e^{s' \alpha})_0)$ for certain $m_1\in M(1)_{r_1}$, $m_2\in M(1)_{r_2}$, and $r_1,r_2>1$, satisfying $s\alpha,s'\alpha,\alpha/p,\alpha/p'\in L^\circ$, and only certain values of $p,p'$ and $\langle\gamma,\alpha\rangle= p(1-r_1)-p'$.\\
	
	And we note that above, $\gamma\in L^\circ$ selects the Virasoro element $\omega_\gamma$ for $V_L$. 
\end{thm}

\begin{proof}
First note that condition $(ii)$ in Definition \ref{screening pair definition} and equation (\ref{eab}) imply that screening pairs with screening operators coming from weight $1$ primary vectors of the form $m\otimes e^{\alpha }$ are of the form $((m_1\otimes e^{-s\alpha })_0, (m_2\otimes e^{s'\alpha })_0)$ for some $s,s'\geq 0$, with $s\alpha,\,s'\alpha\in L^\circ$, and $m_1,m_2\in M(1)$.

If $\alpha=0$, $h_1=\beta_1(-1)\mathbf{1}$, and $h_2=\beta_2(-1)\mathbf{1}$, with $\beta_1 \beta_2\in \h$, then $(h_1)_0$ and $(h_2)_0$ are nilpotent only if $\beta_1= \beta_2 = 0$ and thus there are no screening pairs of this form.  

 Now assume $\alpha\neq 0$. Without loss of generality, we can assume the screening operators are of the form $(m_1\otimes e^{-\alpha/p })_0$ for $\alpha\in L$ satisfying $\alpha/p\in L^\circ$ with $p\in \Z_+$ and $m_1\in M(1)_{r_1}$. (For instance, replacing $\alpha$ with $-\alpha$, we can assume $p\in \Z_+$). For now assume that $\alpha\in L^{(0)}$ and write $\langle \alpha,\alpha\rangle = p q$ for $q\in \Z_+$. For $(m_1\otimes e^{-\alpha/p })_0$ to be a screening operator arising from a weight 1 primary vector as in Definition \ref{screening pair definition}, we need
\[  \mbox{wt} (m_1\otimes e^{-\alpha/p }) \ = \ r_1+\frac{q}{2p}+\frac{1}{p}\langle \alpha,\gamma \rangle \ = \ 1, \] 
 where in the first equality we used (\ref{grading}), giving the restriction $\langle\gamma,\alpha\rangle= p-q/2-p r_1$ for $\omega_\gamma$.  Recall that $\gamma\in \bar{L}^\circ$, i.e. $\gamma$ is in the extended dual space (\ref{extended dual}), and so $q$ must be even for $\alpha\in L^{(0)}$. Set $q=2p'$ for $p'\in \Z_+$, and then $\langle\alpha,\alpha\rangle =2 p p'$. If instead $\alpha\in L^{(1)}$ the rest of the proof follows by replacing $p'$ with $p'/2$, as we show below. 
 
 In order to have a screening pair, we need an additional element $m_2\otimes e^{m \alpha /2p p' }$ with $m \alpha/ 2p p' \in L^\circ$, and $m_2\in M(1)_{r_2}$ such that $\mbox{wt} (m_2\otimes e^{m\alpha /2p p' })=1$. Note that condition $(ii)$ in Definition \ref{screening pair definition} restricts $m$ to be positive. Using the requirement that $\langle \gamma,\alpha\rangle= p(1-r_1)-p'$, we have the condition
 \begin{align*} 
 1=\mbox{wt} (m_2 \otimes e^{\frac{m}{2p p' }\alpha}) &=r_2+\frac{m^2}{4p p' }- \frac{m}{2p p' } (p-p'-p r_1) \\ &= r_2 + \frac{m}{2p p' }\left(\frac{m}{2} +p(r_1-1)+p'\right).
\end{align*}   Multiplying both sides by $4p p'$ and rearranging terms gives 
\begin{align} \label{quadratic}  m^2+ 2m (p(r_1-1)+p')+4p p'(r_2-1)=0.
\end{align} 

We will discuss five possible cases below which result in determining when $m_1\otimes e^{-\alpha/p }$ and $m_2 \otimes e^{m \alpha /2p p' }$ are weight $1$ vectors.

{\it Case 1,} when $r_1,r_2 >0$.  In this case,  the conditions $p,p',m,r_1,r_2\in \Z_+$ imply that the left hand side of the equation $(m/2p p' )(m/2 +p(r_1-1)+p')=1-r_2$ is strictly greater then zero while the right hand side is less then or equal to zero.  Thus no solutions exist.

{\it Case 2,}  when $r_1=0$ and $r_2= 0$.  Viewing (\ref{quadratic}) as a quadratic polynomial in $m$, 
the discriminant is $4(p'-p)^2+16p p'=4(p+p')^2$, and so $m=2p$, leading to weight $1$ vectors as in {\it(i)}.

{\it Case 3,} when $r_1= 0 $ and $r_2=1$. 
 Viewing (\ref{quadratic}) as a quadratic polynomial in $m$, the discriminant is $4(p-p')^2$ and so $m= 2(p-p')$ with $p>p'$, leading to weight $1$ vectors as in {\it(ii)}.

{\it Case 4,} when $r_1= 1 $ and $r_2=0 $.  The discriminant $D$ satisfies  $D/4=p'^2 + 4p p'$. Since this is linear in $p$, all solutions are obtained by setting $p= (r^2-p'^2)/4p'$ for any $r\in \Z_+$ such that $(r^2-p'^2)/4p'$ is a positive integer and $r> p'$. Then $m= -p'+ r$, leading to weight $1$ vectors as in ${\it(iii)}$.

{\it Case 5,}  when $r_1=0,\, r_2>1 $ or  $r_1> 0,\,r_2=0 $.
 If $r_1=0,\, r_2>1 $, we have that $D/4= (p'-p)^2 + 4p p'(1-r_2)$, and if $r_1> 0,\,r_2=0 $, we have that $D/4= ((r_1-1)p+p')^2 + 4p p'$.  These cases only give a perfect square for certain values of $p,p'$, and amount to solving the quadratic Diophantine equation $D/4= r^2$ for some $r\in \Z$, leading to weight $1$ vectors as in $(iv)$.

Now let us return to the assertion that it is enough to assume $\alpha\in L^{(0)}$. Suppose not, and consider screening operators of the form $(m_1\otimes e^{-\alpha/q })_0$ for $\alpha\in L$ satisfying $\alpha/q\in L^\circ$ with $q\in 
\Z_+$ and $m_1\in M(1)_{r_1}$. Then it follows that $\langle\alpha,\alpha\rangle= q q'$ for some $q'\in \Z_+$. Then for $(m_1\otimes e^{-\alpha/q })_0$ to be a screening operator, we need wt$(m_1\otimes e^{-\alpha/q })_0= r_1+ q' /2q- \langle\gamma,\alpha/q\rangle=1$, giving the restriction $\langle\gamma,\alpha\rangle= q-q'/2-q r_1\notin \Z$ for defining  $\omega_\gamma$.
 Let $\alpha'=2\alpha$, $p=2q$, and $p'=q'$. Then $\langle\alpha',\alpha'\rangle= 2 p p'$, $\alpha'\in L^{(0)}$, and $\alpha'/p\in L^\circ$.  Then we can just view screening pairs involving $\alpha$ as screening pairs involving $\alpha'$.

Next we verify that the pairs in cases $(i)-(iii)$ of the proposition give primary vectors. 
Case $(i)$ in the proposition is clear.  For case $(ii)$, using the fact that  $[L_\gamma(m),\beta(n)]= -n\beta(m+n)- m(m+1)\langle\beta,\gamma\rangle \delta_{m,-n}$, we have that 
\begin{align*} L_\gamma(1) \beta(-1)e^{(-p'+p)\alpha/pp'}& = (\beta(0)-2\langle \beta,\gamma\rangle)e^{(-p'+p)\alpha/pp'}\\
&= \langle \beta,(-p'+p)\alpha/pp'- 2\gamma\rangle e^{(-p'+p)\alpha/pp'},
\end{align*} so we need that $\langle \beta,(-p'+p)\alpha/pp'- 2\gamma\rangle=0$. Similarly, for case $(iii)$, we have that 
$L_\gamma(1)\beta(-1)e^{-\alpha/p}=\langle \beta,-\alpha/p-2\gamma\rangle^{-\alpha/p}$, so we need that $\langle \beta,-\alpha/p-2\gamma\rangle=0$.
   
Finally, we need to verify that the three classes of pairs of screening operators $(i)-(iii)$ in the proposition do indeed satisfy the first condition in Definition \ref{screening pair definition} of screening pair.  For {\it(i)} of the proposition, first recall that $(L(-1)v)_0= 0$ for $v\in \tilde{V}$. The Jacobi identity (\ref{egjacobi0}) and $L_\gamma(-1)$-derivative property with $L_\gamma(-1)e^{\alpha}=\alpha(-1)e^{\alpha}$ gives
\begin{align*} 
[ e^{-\alpha/p }_0, e^{\alpha/p' }_0] &= (e^{-\alpha/p  }_0 e^{\alpha/p' })_0 \ = \ \left(\frac{-\alpha(-1)}{p}e^{(p-p')\alpha/pp' }\right)_0 \\ 
&= \left(\frac{-p'}{p-p'}  L_\gamma(-1)e^{(p-p')\alpha/pp' }\right)_0 \ = \ 0.
\end{align*}  
This verifies that  $(e^{-\alpha/p }_0, e^{\alpha/p' }_0)$ is a screening pair. 
 
Moreover, since $\langle (-p'+p)\alpha/pp', -\alpha/p\rangle= -2(-p'+p)/p$, and using the fact that $e_0^{-\alpha/p}$ is a derivation, as operators on  $V_L$ we have that  
\begin{align} \label{4.8}
[e^{-\alpha/p}_0 &,(\beta(-1) e^{(-p'+p)\alpha/pp'})_0]\\
& =\left(e^{-\alpha/p}_0(\beta(-1)e^{(-p'+p)\alpha/pp'})\right)_0 \nonumber\\
& =\left (((e^{-\alpha/p}_0\beta(-1)\mathbf{1})_{-1}+\beta(-1)e^{-\alpha/p}_0)e^{(-p'+p)\alpha/pp'}\right)_0\nonumber\\
&=\left((-\langle\beta,-\alpha/p\rangle e^{-\alpha/p}_{-1}+\beta(-1)e^{-\alpha/p}_0e^{(-p'+p)\alpha/pp'}\right)_0\nonumber\\
&=\left(\langle\beta,\alpha/p\rangle S_{2(-p'+p)/p}(-\alpha/p)e^{(p-2p')\alpha/p p'}\right)_0 \nonumber \\
&\quad +\left(\beta(-1)S_{2(-p'+p)/p-1}(-\alpha/p)e^{(p-2p')\alpha/p p'}\right)_0, \nonumber
\end{align} 
which is $0$ unless possibly when $2(p-p')/p \in \N$.  But recall $p>p'$ in this case, and so
 we can conclude that 
 \[  [(\beta(-1)e^{(-p'+p)\alpha/pp'})_0,  e^{-\alpha/p}_0]=0  \]
 unless possibly when $p=2p'$. In the case $p=2p'$ we have that (\ref{4.8}) is $-(\langle\alpha,\beta\rangle/p^2)\alpha(0)+\beta(0)$, which is nonzero unless $\beta\in \mbox{Span}_\C \{\alpha\}$. But then $(\beta(-1) e^{(-p'+p)\alpha/pp'})_0\in \mbox{Span}_{\C}\{ (L_\gamma(-1)e^{(-p'+p)\alpha/pp'})_0 \} = 0$. Thus it is necessary to have $\beta\notin \mbox{Span}_\C \{\alpha\}$, so that the second screener in case $(ii)$ of the proposition is nonzero, and then me must have $p\neq 2p$ to have condition $(i)$ in Definition \ref{screening pair definition} satisfied.

 Similarly, with $\langle (-p'+r)\alpha/2pp',-\alpha/p\rangle= (p'-r)/p$, we have that
 \begin{align} \label{4.9} 
 [&(e^{(-p'+r)\alpha/2p p'})_0 ,(\beta(-1)e^{-\alpha/p})_0]\\
& = ( -\langle \beta,(-p'+r)\alpha/2p p'\rangle  S_{(r-p')/p}((-p'+r)\alpha/2p p' ) \nonumber \\
&\quad +\beta(-1)S_{(r-p')/p-1}((-p'+r)\alpha/2p p') e^{(-3p'+r)\alpha/2p p'} )_0. \nonumber 
 \end{align} 
Note that $(r-p')/p= 4p'/(r+p')$, so we can conclude that 
 \begin{align*} 
 [&(e^{(-p'+r)\alpha/2p p' })_0 ,(\beta(-1)e^{-\alpha/p})_0]=0,
 \end{align*} 
 unless $r= 4p'/a-p'$ for some positive integer $a$ that divides $4p'$. Since $r>p'$, this can only happen when $r=3p'$, in which case (\ref{4.9}) is again $-(\langle\alpha,\beta\rangle/p^2)\alpha(0)+\beta(0)$. Thus again we only have a nonzero screener if  $\beta\notin \mbox{Span}_\C \{\alpha \}$ and also $r\neq 3p'$.
This verifies that condition $(i)$ in the definition of screening pair is satisfied for the pairs of type {\it(ii)} and {\it(iii)} in the proposition.
This completes the proof. 
\end{proof}

For the rest of this paper, we will focus on screening pairs of the form $(e^{-\alpha/p}_0,e^{\alpha/p'}_0)$ with $\langle\alpha,\alpha\rangle=2p p'$. Now let us apply the Theorem \ref{thm for screening pairs} to the case when $L$ is a rank one lattice $L$. First suppose $L=\sqrt{2 p p'}\Z$, so that $L$ is an even lattice. If we set $n=-2p'$, and $\gamma=  (p-p') \alpha /2pp'$, then $\omega_\gamma$ gives a central charge of $c_{p,p'}=1- 6 (p-p')^2/pp'$, and $(e^{-\alpha/ p}_0,e^{\alpha/p'}_0)$ is a screening pair of type $(p,p')$. In the case that $p'=1$, then $\mathcal{W}(p)=\mathcal{W}(p,1)= \mbox{ker} \, e^{-\alpha/p}_0$ is the kernel of the long screening operator. In the case that $p,p'>1$ and $p,p'$ are relatively prime, the vertex operator algebra $\mbox{ker} \, e^{-\alpha/p}_0\cap   \mbox{ker} \, e^{\alpha/p'}_0$ is denoted by $\mathcal{W}(p,p')$, see e.g. \cite{TW2}.

\section{Screening pairs of type $(p,p')$ for a rank $d$ lattice}

Our next goal is to give a characterization of elements in the lattice that can give rise to screening pairs in $V_L$.  Consider the screening pairs of type $(e^{-\alpha/p}_0,e^{\alpha/p'}_0)$ with $\alpha\in L^{(0)}$ and $\langle\alpha,\alpha\rangle=2p p'$. In this case it follows from Theorem \ref{thm for screening pairs} that both $\alpha/p, \alpha/p'\in L^\circ$.  When $p,p'$ are relatively prime, we have that $2\alpha/\langle \alpha,\alpha\rangle \in L^\circ$.  This leads to the following definition.

\begin{defn} \label{screener defn}
	The set of {\em screeners}  or {\em screening momenta} in $L$, denoted by $\Phi$, are defined as \\
	$$\Phi= \left\{\beta \in L^{(0)} \,\Big|\,  \beta \notin 2L \mbox{ and } \,\frac{2\beta}{\langle \beta,\beta\rangle}\in L^\circ \right\}.$$ \\
\end{defn} 

We will denote by $\Z\Phi=\mbox{Span}_\Z \Phi$ the lattice generated by the screeners.
We will see shortly that if $L$ has rank $d>1$ and is indecomposable then the condition $\beta \notin 2L$ is redundant in the definition above.
Notice that $\Phi$ is finite, and trivially contains all roots, i.e., those elements $\alpha\in L^{(0)}$ with $\langle\alpha,\alpha\rangle=2$. 

Before citing the main properties of screeners, we need the following lemma.

\begin{lem} \label{lem .2}
	If $\alpha\in L^{(0)}$, $\alpha\neq 0$, satisfies  $2\alpha/\langle \alpha,\alpha\rangle \in L^\circ$ for a lattice $L$ of rank $d>1$ such that $L$ does not have an orthogonal direct sum decomposition of the form $L=L_1\oplus \Z\gamma$ for some lattice $L_1$ of rank $d-1$ and $ \alpha\in \Z\gamma$, then $\alpha \notin 2L$, implying $\alpha$ is a screener. In addition, there exists a $\beta\in L$ such that $\langle 2\alpha/\langle\alpha,\alpha\rangle, \beta\rangle \in 2\Z+1$.
\end{lem}

\begin{proof} Let $\alpha\in L$ satisfy $2\alpha/\langle \alpha,\alpha\rangle \in L^\circ$ and assume there is no orthogonal decomposition of the form $L=L_1\oplus \Z\gamma$ with $\alpha\in \Z\gamma$. Suppose $\alpha\in 2L$ and write $\alpha=2\beta$ for $\beta\in L$. Then rewriting the condition $2\alpha/ \langle \alpha,\alpha\rangle\in L^\circ$ gives
 \[ \langle 2\beta,L\rangle =  \langle \alpha,L\rangle\subset \frac{\langle \alpha,\alpha\rangle}{2} \Z =2\langle\beta,\beta\rangle \Z, \] 
 or $ \langle\beta,L\rangle\subset  \langle\beta,\beta\rangle \Z$. By Lemma \ref{first lem}{\it (ii)}, $L$ has an orthogonal decomposition of the form $L=L_1\oplus \Z\beta$ with $\alpha\in \Z\beta$, a contradiction. Thus $\alpha\notin 2L$. 
 
  Now suppose there is no $\beta\in L$ such that  $\langle 2\alpha/\langle\alpha,\alpha\rangle , \beta\rangle \in 2\Z+1$. Then $\langle \alpha/\langle\alpha,\alpha\rangle , L\rangle \in \Z$ and again by Lemma \ref{first lem}{\it (ii)}, $L=L_1\oplus \Z\alpha$, a contradiction. 
\end{proof}

\begin{lem} \label{Z-base with screener}
	Every screening pair of type $(p,p')$ of the form $(e^{-\alpha/p}_0,e^{\alpha/p'}_0)$ with $p,p'$ relatively prime and $\langle\alpha,\alpha\rangle=2pp'$ gives rise to a screener $\alpha$. Conversely, every screener $\alpha$ with $\langle\alpha,\alpha\rangle=2pp'$ ($p,p'$ not necessarily relatively prime) gives rise to a screening pair $(e^{-\alpha/p}_0,e^{\alpha/p'}_0)$ of type $(p,p')$ with Virasoro element $\omega_\gamma$ where $\gamma\in L^\circ$ satisfies $\langle\gamma,\alpha \rangle= p-p'$. Moreover, there exists a $\Z$-basis containing $\alpha$.
\end{lem}

\begin{proof} If $(e^{-\alpha/p}_0, e^{\alpha/p'}_0)$ is a screening pair, then $\alpha/p, \alpha/p'\in L^\circ$.  If $p,p'$ are relatively prime, then $\alpha/pp'\in L^\circ$, which implies $2\alpha/\langle\alpha,\alpha\rangle \in L^\circ$. By Lemma \ref{lem .2}, we know $\alpha\notin 2L$ except in the case $L$ has an orthogonal decomposition $L=L_1\oplus \Z\beta$ with $\alpha\in \Z\beta$. In this case, the only possibilities are $\alpha=\beta$ or $\alpha=2\beta$. Suppose it is the latter case. Then $\langle\alpha,\alpha\rangle\in 4\Z$, and so exactly one of $p,p'$ are even (recall $p,p'$ are relatively prime), and by 
Theorem \ref{thm for screening pairs} there exists a $\gamma\in \bar{L}^\circ$ such that $ \langle\alpha,\gamma\rangle=p-p'\in 2\Z+1$. But then $\langle \beta,\gamma\rangle\in \Z+ \frac{1}{2}$, contradicting $\gamma\in \bar{L}^\circ$. Hence $\alpha \notin 2L$ and $\alpha$ is a screener, proving the first statement in the lemma.

For the second statement, we only need to verify such a $\gamma$ exists as in Theorem \ref{thm for screening pairs}. 
Since $\alpha$ is a screener, we have that $\alpha\notin 2L$ by definition, and so $\langle \alpha,L^\circ\rangle \nsubseteq 2\Z$ by Lemma \ref{first lem}$(iii)$.  But $\beta =2\alpha/\langle \alpha,\alpha\rangle\in L^\circ$  satisfies $\langle \alpha,\beta\rangle=2$.  Thus we can find an element $\bar{\gamma}$ with $\langle\alpha,\bar{\gamma}\rangle=1$, and setting $\gamma=(p-p')\bar{\gamma}$ gives the required element in $L^\circ$. By Lemma \ref{first lem}$(i)$, since such a $\bar{\gamma}$ exists, there exists a $\Z$-basis containing $\alpha$, proving the last statement. 
\end{proof}

We now give a few lemmas that give a characterization of $\langle\cdot,\cdot \rangle$ restricted to $\Phi$.

\begin{lem}  \label{interaction of bilinear form}
	Let $\alpha,\beta\in \Phi$. Without loss of generality suppose $$\langle \alpha,\alpha\rangle \leq  \langle \beta,\beta\rangle.$$
	Then one of the following holds:
	\begin{align*} 
	(i) \quad   &\langle \alpha,\beta\rangle=  0 . \\
	(ii)\quad  &\langle \alpha,\beta\rangle= \pm\langle\beta,\beta\rangle  \mbox{ and } \alpha = \pm \beta. \\
	(iii)\quad  &\langle \alpha,\beta\rangle= \pm \frac{1}{2}\langle\beta,\beta\rangle \mbox{ and } \langle \beta,\beta\rangle=a \langle \alpha,\alpha\rangle, \mbox{ with } a=1,2,3.\\
	\end{align*} 
\end{lem}

\begin{proof}  
If $\alpha,\beta\in\Phi$, then it follows that $\langle 2\beta/\langle \beta,\beta\rangle,\alpha\rangle, \langle 2\alpha/\langle \alpha,\alpha\rangle,\beta\rangle \in \Z$, which implies that if $l=\mbox{lcm}(\langle \beta,\beta\rangle,\langle \alpha,\alpha\rangle)$, then $\langle\alpha,\beta\rangle \in \frac{l}{2}\Z.$  Let $m,n\in \Z_+$ such that $l= m\langle\alpha,\alpha,\rangle= n \langle\beta,\beta\rangle$. Then $\mbox{gcd}(m,n)=1$ and $m\geq n$. Let $k\in\Z$ be such that $\langle\alpha,\beta\rangle=k l/2$. Then
\begin{align*}
\langle \alpha,\beta\rangle=\frac{k m}{2}\langle\alpha,\alpha\rangle= \frac{k n}{2}\langle\beta,\beta\rangle
\end{align*}
and $\langle \alpha,\alpha \rangle = (n/m)\langle \beta,\beta\rangle$. By the Cauchy-Schwartz inequality, it follows that 
\begin{align}\label{cs} 
0\leq \langle\alpha,\beta\rangle^2 \leq \langle\alpha,\alpha\rangle\langle\beta,\beta\rangle=\frac{n}{m}\langle\beta,\beta\rangle^2,
\end{align} 
with the second inequality being equality if and only if $\alpha=r \beta$ for some $r\in\mathbb{R}$.  
This implies
\[0 \leq \frac{k^2 n^2}{4}\langle \beta,\beta\rangle^2 \leq \frac{n}{m}\langle\beta,\beta\rangle^2, \]
i.e. $0\leq k^2 m n \leq 4$. When $m=4$, then the second inequality in (\ref{cs}) is equality, and so $\beta=2\alpha$, which contradicts the assumption that $\beta\in\Phi$. The cases $k=0$, $k=\pm 2$, and $k=\pm 1$  lead to the statements {\it(i)}, {\it(ii)}, and {\it(iii)} in the lemma, respectively, where in case {\it(ii)} we also have $\alpha=\pm \beta$, since $k=\pm 2$ gives equality in the second inequality in (\ref{cs}).  
\end{proof}

The next lemma describes the internal structure of $\Phi$. Notice how some of the properties mirror that of a root system, but the set of screeners is at this point theoretically possibly larger.

\begin{lem} \label{basis screener results} Let $\alpha,\beta\in \Phi .$ Without loss of generality suppose $\langle \alpha,\alpha\rangle \leq  \langle \beta,\beta\rangle$.  Then the following hold:

	(i) The only other multiple of $\alpha$ that is a screener is $-\alpha$.  No element in $nL$ for $n>1$ is a screener. 

	(ii) If $\langle\alpha,\beta\rangle=-\langle \beta,\beta\rangle/2$, then $\alpha+\beta\in \Phi.$ 

	(iii) Suppose $\langle \alpha,\beta\rangle= 0$.  Then $\alpha\pm\beta\in \Phi$ if and only if the following 3 properties hold: $\langle\alpha,\alpha\rangle=\langle\beta,\beta\rangle$, $\langle \alpha\pm \beta,L\rangle \subset \langle\alpha,\alpha\rangle\Z$, and $\alpha\pm \beta\notin 2L$.  
\end{lem}

\begin{proof} For $a\in \Z$ and $\alpha\in L$, suppose $a\alpha\in \Phi$. Then we have that $\langle 2a\alpha/\langle a\alpha,a\alpha\rangle,\alpha\rangle = 2/a$. By Definition \ref{screener defn}, this implies $2/a\in\Z$ and also that $\alpha\notin 2L$. Thus we must have that $a=\pm 1$, proving $(i)$.

For $(ii)$,  let $\langle\alpha,\beta\rangle=-\langle\beta,\beta\rangle/2$ and $\langle\beta,\beta\rangle=a\langle\alpha,\alpha\rangle$, where by Lemma \ref{interaction of bilinear form}, we must have that $a=1,2,3$.  Then $\langle\alpha+\beta,\alpha+\beta\rangle=\langle\alpha,\alpha\rangle$, so that $2(\alpha+\beta)/\langle\alpha+\beta,\alpha+\beta\rangle \in L^\circ$. Since $\langle \alpha+\beta,\alpha\rangle=(2-a)\langle\alpha+\beta,\alpha+\beta\rangle/2$, there is no orthogonal decomposition of the form $L=L_1\oplus \Z(\alpha+\beta)$ when $a=1,3$.  When there is no orthogonal decomposition, Lemma \ref{lem .2} implies that $\alpha+\beta \notin 2 L$. On the other hand, if there is an orthogonal decomposition in the case $a=2$, then by Lemma \ref{first lem}, again $\alpha+\beta\notin 2L$, proving $(ii)$.

For $(iii)$,  first suppose that $\langle\alpha,\beta\rangle=0$. In order for $\alpha\pm \beta\in \Phi$, we must have in particular that $2\langle (\alpha\pm \beta)/(\langle\alpha,\alpha\rangle + \langle\beta,\beta\rangle),\alpha\rangle \in \Z$ and 
$ 2\langle (\alpha\pm \beta)/(\langle\alpha,\alpha\rangle + \langle\beta,\beta\rangle),\beta\rangle\in \Z$. This can happen if and only if $\langle\alpha,\alpha\rangle=\langle\beta,\beta\rangle$. In this case, $\alpha\pm \beta $ can be a screener if and only if $\langle \alpha\pm \beta, L\rangle\subset \langle\alpha,\alpha\rangle\Z$ and $\alpha\pm \beta\notin 2L$, proving $(iii)$. 
\end{proof}

\begin{rema}
Notice that Lemma \ref{interaction of bilinear form} shows that the possibles angles between any two vectors in $\Phi$ are the same as the possible angles between root vectors in a root system for a finite dimensional Lie algebra \cite{H}, and Lemma \ref{interaction of bilinear form} part $(i)$ and $(ii)$ also gives properties identical to those for a root system. 
\end{rema}

Viewing $\Phi$ as a subset of $\C\otimes_\Z L$, we call a linearly independent subset $\overline{\Phi}$ of $\Phi$ of size $r$ a {\it screening system of rank $r$}.  Notice that unlike the case with just one screener, it is not necessarily true that there exists a Virasoro element $\omega_\gamma$ with a $\gamma\in L^\circ$ simultaneously satisfying $\langle \gamma,\alpha_i\rangle=p_i-q_i$ when $\langle\alpha_i,\alpha_i\rangle=2p_i q_i$ and $\alpha_i\in \Phi$. For instance, in some cases such a choice of conformal element $\omega_\gamma$ may admit $L(0)$-eigenvalues that do not satisfy the $\Z$-grading condition on $V^0$ in the definition of vertex operator algebra. In the case that such $\gamma$ exists that does give a $\Z$-gradation on $V^0$, we denote by $(\overline{\Phi},\gamma)$ the screening system $\overline{\Phi}$ along with such a $\gamma$.  

Of special importance to us are lattices which admit screening systems of size $d$, the rank of the lattice, which we will refer to as a {\it full screening system} for $L$. In this case we can write $\overline{\Phi}=\{\alpha_1,\dots,\alpha_d\}$, and with $\hat{L}=\Z\overline{\Phi}$, we have the standard corresponding $\Z$-basis $ \{\alpha_1^\circ,\dots,\alpha_d^\circ\}$ for $\hat{L}^\circ$. Then such a $\gamma$, if it exists in $L^\circ$ as well, is uniquely determined, and given by $\sum_{i=1}^d (p_i-q_i) \alpha_i^\circ$.

Now we introduce some more notation. Suppose $\Z\Phi$ is of rank $r\leq d$. Then there is a maximal sublattice $L_1\subset L$ of rank $r$ such that $\Z\Phi\subset L_1$.  Let $B_1$ be a $\Z$-basis for $L_1$, and $A$ be a $\Z$-basis for $\Z\Phi$.     By the maximality of $L_1$, $L=L_1\oplus L_2$ as a finitely generated abelian group for some sublattice $L_2$ of rank $d-r$, and a basis $B$ can be chosen for $L$ such that $B=B_1\cup B_2$, where $B_2$ is a $\Z$-basis for $L_2$. Let $L_\Phi= \Z(A\cup B_2)$, that is, \begin{align} \label{splitting up L} 
L_\Phi=\Z\Phi\oplus L_2
\end{align}  as a finitely generated abelian group.
Similarly, we can define $\Z\overline{\Phi}$ and $L_{\overline{\Phi}}$  for a screening system $(\overline{\Phi},\gamma)$.

\begin{lem} \label{divides deg G}
	Let $G_L$ be a Gram matrix for $L$ with respect to some ordered $\Z$-basis, and $S\subset \Phi$ a $\Z$-basis for $\Z\Phi$.  Then $\mbox{lcm}(\langle \alpha,\alpha\rangle/2|\,\alpha\in \Phi )$ divides $\mbox{Det}(G_L)$.
	
\end{lem} 

\begin{proof} Recall that $|L^\circ/L|= \mbox{Det}(G_L)$.  Let $\alpha\in \Phi$. Since $2\alpha/\langle\alpha,\alpha\rangle \in L^\circ$, then $2\alpha/\langle\alpha,\alpha\rangle + L$ is an element of $L^\circ/L$ of order $\langle\alpha,\alpha\rangle/2$, since if not, then $\alpha\in n L$ for $n>1$, contradicting Lemma \ref{basis screener results}$(i)$. It follows that $\mbox{lcm}(\langle \alpha,\alpha\rangle/2|\,\alpha\in \Phi )$ divides $\mbox{Det}(G_L)$.  
\end{proof}

\begin{lem}  \label{some lem}
	We have the (rational) sublattice embedding and properties of quotients:

	(i) $2L\subset L_\Phi \subset L\subset L^\circ \subset (L_\Phi)^\circ\subset (2L)^\circ$;

	(ii) If $L\neq L_\Phi$, then $|(L_\Phi)^\circ/L_\Phi|  \mbox{ divides } 4^r |L^\circ/L|$ for some $1\leq r\leq d$. 
\end{lem}

\begin{proof} Note that if $A\subset B$, then $B^\circ\subset A^\circ$, and thus immediately we have $L_\Phi\subset L \subset L^\circ\subset L_{\Phi}^\circ$. Recalling the definition of $L_\Phi$ in (\ref{splitting up L}), we first show this when $\Z \Phi= L_\Phi$.  We have that for $\alpha$ a screener, $\langle \alpha/\langle\alpha,\alpha\rangle ,\beta\rangle\in \Z$ for any $\beta\in 2L$. So by Lemma \ref{lem .2} and using the definition $L_\Phi$ in (\ref{splitting up L}), we have that $(L_\Phi)^\circ=(\Z\Phi)^\circ\subset \Z\{\alpha/\langle\alpha,\alpha\rangle \,|\,\alpha\in \Phi\}$ with equality if and only if $L$ is a direct sum of rank $1$ lattices.  Therefore, $(L_\Phi)^\circ\subset \Z\{\alpha/\langle\alpha,\alpha\rangle \,|\,\alpha\in \Phi\}\subset (2L)^\circ$, proving $(i)$ in the case that $L_{\Phi}=\Z\Phi$. Then when $\Z\Phi \neq L_{\Phi}$, the construction of $L_\Phi$ in (\ref{splitting up L}) implies that $2L\subset L_\Phi$, and the rest of $(i)$ follows.

For $(ii)$, let $L_1$ and $L_2$ be lattices of the same rank with $L_1 \subset L_2$,  and let $A$ be a change of $\Z$-basis matrix from a $\Z$-basis for $L_1$ to a $\Z$-basis for $L_2$. Let $G_1$, $G_2$ be the Gram matrices with respect to these choices of $\Z$-base. Then it follows that $Det(G_2)=Det(A)^ 2 Det (G_1)$. 
Part $(ii)$ now follows from $(i)$.  
\end{proof}

\begin{rema}
	Lemma \ref{some lem} says nothing about the value of $|L_{\Phi}|/|L_{\overline{\Phi}}|$, which does not have to necessarily be a power of $4$. It is only necessarily true that $2L\subset L_{\overline{\Phi}}$ when $\overline{\Phi}$ is a $\Z$-basis for $L_\Phi$ as opposed to an arbitrary linearly independent subset.
\end{rema}
The next corollary gives another way of looking at lattices with screeners, and follows immediately from Lemma \ref{some lem}.
\begin{cor}
	Let $L$ be a positive definite integral lattice with at least one screener.  Then $L$ is an (integral) sublattice of the dual of a lattice which as a finitely generated group is of the form $L_\Phi= \Z\Phi \oplus L'$, where $\Phi$ is the set of screeners for $L$ and $L'$ is another lattice. 
\end{cor}

The problem of finding screeners for an arbitrary lattice $L$ can be solved in three steps: The first is determining $\Z\Phi$ as a sublattice of $L$; the second is determining all screeners for $\Z\Phi$; and the third is determining which screeners for $\Z\Phi$ are screeners for $L$. The first step is the hardest and depends on the internal properties of $L$. 
The second step will be done below, and the third is a simple check when using Lemmas \ref{interaction of bilinear form} and \ref{basis screener results}. 

We now proceed to solve the second step, which amounts to studying lattices that are generated by screeners. For instance, an example of a lattice generated by screeners would be $L=\sqrt{p}K$, where $K$ is a simply laced root lattice. Note that there is an obvious bijection between screeners for $K$ and screeners for $\sqrt{p}K$, so in the simply laced case, we only need to study $\Phi$ for the simply laced root lattices $K$.  In fact, the following theorem shows that in the positive definite even case, lattices generated by a screening system must have a basis of screeners that give rise to an orthogonal direct sum of such (possibly rescaled) simply laced root lattices.    

\begin{thm} \label{L-even-classification-theorem}
	Let $L$ be a positive definite even lattice generated by a screening system $\overline{\Phi}$.  Then $L$ has a basis of screeners such that $L$ is an orthogonal direct sum of lattices $\sqrt{p_i}K_i$, where the $K_i$ are one of the simply laced root lattices $A_n, D_n, E_6, E_7, E_8$, for $n \in\mathbb{Z}_+$.
\end{thm}

\begin{proof} Let $B= \{u_1,u_2,\dots,u_d\}\subset \Phi$ be a $\Z$-basis of screeners for $L$, ordered in such a way that $\langle u_i,u_i \rangle \leq \langle u_{i+1},u_{i+1} \rangle$. Suppose there is some (smallest) $l>1$ such that $\langle u_1,u_l \rangle < \langle u_{l},u_{l} \rangle$ and $\langle u_{1},u_{l} \rangle\neq 0$. Then from Lemma \ref{interaction of bilinear form}, $\langle u_1,u_1 \rangle$ divides $\langle u_{l},u_{l} \rangle$ and $\langle u_1,u_{l} \rangle= \pm \langle u_{l},u_{l} \rangle/2$. Then we can form a new $\Z$-basis $\hat{B}$ of screeners by replacing $u_{l}$ with $\hat{u}_{l}=u_{l}\mp  u_1$.

Note that $\hat{u}_{l}$ is also a screener, and that $\langle u_{1},u_{1}\rangle =\langle \hat{u}_{l},\hat{u}_{l}\rangle$. We can then reorder $\hat{B}=\{u_1,\dots,u_n\}$ so that it takes the form $\langle u_i,u_i \rangle \leq \langle u_{i+1},u_{i+1} \rangle$ as before for $u_i\in B$.  Now repeat this process again for $u_1$ until there is no smallest $l>1$ as above, and likewise for $u_2,\dots u_{d-1}$, and one arrives at a $\Z$-basis of screeners satisfying $\langle u_i,u_j\rangle=0$ if $\langle u_i,u_i\rangle\neq \langle u_j,u_j\rangle$. Then $L$ decomposes as an orthogonal direct sum of lattices $\bigoplus_{k = 1}^m L_k$, for some $m\in \mathbb{Z}_+$ where each $L_k$ is generated by screeners of the same norm squared equal to $2p_k$ for some positive integer $p_k$.  It follows from the theory of root lattices \cite{H} and Lemmas \ref{interaction of bilinear form} and \ref{basis screener results}  that each $L_k$ is of the form $L_k =\bigoplus_{j=1}^r \sqrt{p_k} K_{k_j}$, for some $r \in \mathbb{Z}_+$, where $K_{k_j}$ are simply laced root lattices. 
\end{proof}

We are now in the position to completely classify all rank 2 positive definite integral lattices, which we do in Section 5.1 below, as well as all positive definite even lattices generated by a screening system, which we do in Section 5.2 below, along with the application of Theorem \ref{L-even-classification-theorem}, by classifying all screeners in a simply laced root lattice.

\subsection{Classification of screeners for rank 2 positive definite integral lattices}

\begin{thm} \label{classification rank 2}  Let $p\in\Z_+$.  Every rank 2 positive definite integral lattice with at least one screener has a $\Z$-basis with Gram matrix of the following two types: 

   Type 1 : $\begin{pmatrix} 2p & 0  \\ 0 & m \end{pmatrix}$ with $m\neq 2p$, $m\in \Z_+$.
 
  Type 2:   $\begin{pmatrix} 2p &-p  \\ -p & m \end{pmatrix}$ with $m\geq p$.

The Type 2 lattices coincide with the set of rank 2 positive definite even integral lattices that are generated by elements of the same norm.  

The set of screeners $\Phi$ for these two types of rank 2 positive definite integral lattices are classified as follows:\\

Type 1:   If $m\in 2\Z+1$, then $\Phi=\{\pm\alpha_1\}$, and if $m\in 2\Z$, then $\Phi= \{\pm\alpha_1, \ \pm\alpha_2\}$.\\

Type 2:  Set $\bar{\Phi}= \{\pm \alpha_1,\ \pm (\alpha_1+2\alpha_2)\}$.   There are 3 subtypes:\\

2(a) If $m\neq2p$ and $m\neq p$, then $\Phi= \bar{\Phi}$ and $$\Z\Phi=\Z\bar{\Phi}= \Z\alpha_1\oplus \Z(\alpha_1+2\alpha_2)= \sqrt{2p}\Z\oplus \sqrt{4m-2p}\Z.$$ 

2(b) If $m= p$, then  $\Z \Phi = L=\Z(\alpha_1+\alpha_2)\oplus  \Z\alpha_2 = (\sqrt{p}\Z)^2 = (\sqrt{p/2} A_1)^2$, and
$\Phi=\bar{\Phi}\cup \{\pm\alpha_2, \ \pm(\alpha_1+\alpha_2)\}$, the total set of screeners giving the roots of the classical root system $B_2 = C_2$, rescaled by $\sqrt{p/2}$.   \\

2(c) If $m=2p$, then $\Z \Phi = L=\sqrt{p}A_2=\sqrt{p}D_2$, and $\Phi=  \bar{\Phi}\cup  \{ \pm \alpha_2, \ \pm(\alpha_1+\alpha_2), \ \pm( \alpha_1-\alpha_2), \ \pm(2\alpha_1+\alpha_2)\}$, the total set of screeners giving the roots of the exceptional root system $G_2$, rescaled by $\sqrt{p}$.
\end{thm}

\begin{proof}
 By Lemma \ref{Z-base with screener}, every rank 2 positive definite integral lattice with at least one screener $\alpha_1$ satisfying $\langle\alpha_1,\alpha_1\rangle=2p$ for some $p\in \Z_+$ has a $\Z$-basis $\{\alpha_1,\alpha_2\}$. Then from Definition \ref{screener defn}, $\langle \alpha_1,L\rangle\in p\Z$, and so the $\Z$-basis $\{\alpha_1,\alpha_2\}$ for $L$ has a Gram matrix of the form $G=\begin{pmatrix} 2p & -np  \\ -np & c \end{pmatrix}$ 
 for some $n\in \Z$ and $c\in \Z_+$ satisfying $\mbox{Det}(G)>0$, which implies $2c-n^2p> 0$. 

If $n\in 2\Z$, we can replace $\alpha_2$ with $\alpha_2+ n\alpha_1/2$, giving a Gram matrix of the form $\begin{pmatrix} 2p & 0  \\ 0 & m \end{pmatrix}$ where $m= c-n^2p/2$, and so $L$ decomposes as $L=\sqrt{2p}\Z\oplus \sqrt{m}\Z$, giving Gram matrices of Type 1, unless $m = 2p$, which we shall see below, can be included in Type 2. 

If $n$ is odd, then we replace $\alpha_2$ with $\alpha_2+ (n-1) \alpha_1/2$ to get a Gram matrix of the form $\begin{pmatrix} 2p & -p  \\ -p & m \end{pmatrix}$ for some $m\in\Z_+$, and $\mbox{Det}(G)= p(2m-p)$ satisfying $m> p/2$.  If $m\leq p$, we can obtain a new $\Z$-basis $\{\hat{\alpha}_1,\hat{\alpha}_2\}$ by setting $\hat{\alpha}_1=\alpha_1+2\alpha_2$ and $\hat{\alpha_2}=-\alpha_1-\alpha_2$, giving a Gram matrix of the form $\begin{pmatrix} 2p' & -p'  \\ -p' & m \end{pmatrix}$, where $p'=2m-p$ and $m\geq p'$. And so without loss of generality, in the $n$ odd case, we can just consider the Gram matrices satisfying $m\geq p$, in particular the Type 2 Gram matrices in the theorem. This also shows that $\alpha_1+2\alpha_2$ is always a screener in the Type 2 case.  

Performing a change of $\Z$-basis on the Gram matrix $\begin{pmatrix} 2p & -p  \\ -p & m \end{pmatrix}$ by replacing $\alpha_1$ with $\alpha_1+\alpha_2$ gives the Gram matrix  $\begin{pmatrix} m & m-p  \\ m-p & m \end{pmatrix}$, proving that Type 2 lattices coincide with the set of rank 2 positive definite even lattices that are generated by two elements of the same norm.  In particular, this encompasses the case of $m=2p$ from Type 1 by considering the case of $m = p \in 2\Z$. 

This completes the proof of the first two statements of the theorem.  

 Next we determine the screeners for lattices with a Type 1 Gram matrix. Suppose $\alpha$ is a screener and set $\alpha=n_1\alpha_1 +n_2\alpha_2$. Suppose $n_1\neq 0$ and replace $\alpha$ with $-\alpha$ as necessary so that $n_1$ is positive. Then $\langle\alpha_1,\alpha\rangle=2n_1 p $ and  by Lemma \ref{interaction of bilinear form}, we must have that $\langle\alpha,\alpha_1\rangle = \pm p,2p,3p$, so that the only possibility is $\langle\alpha,\alpha_1\rangle=2p$ with $n_1=1$. Then again by Lemma 5.5, we have that $\langle\alpha,\alpha\rangle=4p$. Equating this to $\langle\alpha,\alpha\rangle = 2p + n_2^2 m$, we must have that $n_2 m= 2p/n_2$. From the definition of screeners, $\langle\alpha,L\rangle\in 2p\Z$ and in particular $\langle\alpha,\alpha_2\rangle=  n_2 m\in 2p\Z$. This forces $n_2=\pm1$ and $m=2p$, which is not a Type 1 Gram matrix. Therefore, lattices with Gram matrices of Type 1 have $\Phi=\{ \pm\alpha_1\}$ if $m\in 2\Z+1$ and $\Phi= \{\pm\alpha_1,\pm \alpha_2 \}$ if $m\in 2\Z$. This completes the classification of screeners for Type 1.

Next we determine the screeners for lattices with a Type 2 Gram matrix. Set $\hat{\alpha}_2= \alpha_1+2\alpha_2$.  We already have that $\pm\alpha_1,\pm \hat{\alpha}_2\in \Phi.$  
Suppose $\alpha\in \Phi$ with $\alpha\notin \{\pm\alpha_1,\pm \hat{\alpha}_2\}$. By positive definiteness, at least one of $\langle \alpha_1,\alpha\rangle$  and $\langle\hat{\alpha}_2,\alpha\rangle$ are nonzero. Furthermore, Lemma \ref{interaction of bilinear form} implies that if $\langle \alpha,\alpha_1\rangle$  is nonzero, then two situations can happen:
\begin{align} \langle\alpha,\alpha_1\rangle &=\pm p a_1,\mbox{ and } \langle\alpha,\alpha\rangle= 2 pa_1, \, a_1= 2,3 \label{alpha 1A} \\
  \langle\alpha,\alpha_1\rangle&= \pm p, \mbox{ and } \langle\alpha,\alpha\rangle=2 p/\hat{a}_1,\, \hat{a}_1=1, 2,3 . \label{alpha 1B}
\end{align} 
 Similarly, if  $\langle \alpha,\hat{\alpha}_2\rangle$  is nonzero, then we have the two scenarios
 \begin{align} 
 \langle\alpha,\hat{\alpha}_2\rangle &=\pm (2m-p) a_2,\mbox{ and } \langle\alpha,\alpha\rangle=  2(2m-p) a_2 , \, a_2= 2,3 \label{alpha 2A} \\
\langle\alpha,\hat{\alpha}_2\rangle &=\pm  (2m-p), \mbox{ and } \langle\alpha,\alpha\rangle=2 (2m-p)/\hat{a}_2,\, \hat{a}_2= 1,2,3. \label{alpha 2B}
\end{align} 
On the other hand, if we write $\alpha=n_1 \alpha_2 + n_2 \alpha_2$, then the following three equalities hold
\begin{align} 
\langle \alpha,\alpha_1 \rangle&= (2n_1-n_2)p \label{alpha 3A} \\
\langle \alpha,\hat{\alpha}_2 \rangle&= n_2 (2m-p) \label{alpha 3B}\\
\langle \alpha , \alpha \rangle &= 2n_1^2 p - 2n_1 n_2 p + n_2^2 m .\label{alpha 3C}
\end{align}

Throughout the remainder of this proof, we replace $\alpha$ with $ -\alpha$ when necessary so that $\langle\alpha,\hat{\alpha}_2\rangle\geq 0$.

Assume $\langle\alpha_1,\alpha\rangle= 0$. Then (\ref{alpha 3A}) gives $n_2=2 n_1$.  Lemma \ref{basis screener results}$(i)$ implies that $n_1=\pm 1$, and thus $\alpha= \pm \hat{\alpha}_2$, which contradicts our assumption that $\alpha \notin \{ \pm \alpha_1, \pm \hat{\alpha}_2\}$. On the other hand if $\langle\alpha_2,\alpha\rangle=  0$, then (\ref{alpha 3B}) implies $n_2=0$, and the only possibility is $\alpha=\pm \alpha_1$, again a contradiction.  We can thus assume both $\langle\alpha_1,\alpha\rangle$ and $\langle\hat{\alpha}_2,\alpha\rangle$ are nonzero.

First assume $\langle\alpha,\hat{\alpha}_2\rangle = (2m-p)$. Then (\ref{alpha 2B}) implies $\langle \alpha,\alpha\rangle= 2 (2m-p)/\hat{a}_2$, where $\hat{a}_2 =1,2,3$.  Then (\ref{alpha 3B}) implies $n_2 = 1$, and (\ref{alpha 3A}) implies $\langle \alpha,\alpha_1\rangle= (2n_1-n_2)p= (2n_1- 1) p$, and so by (\ref{alpha 1A}) and (\ref{alpha 1B}),  $n_1$ can only possibly be $0,\,\pm 1,\, 2$.

Suppose $n_1= 0,1$. In either case, by (\ref{alpha 3C}), $\langle \alpha , \alpha \rangle = 2n_1 p (n_1 - 1) + m$, and so we have that $\langle\alpha,\alpha\rangle= m$.  In addition, $n_2 = 1$ and $n_1 = 0, 1$, along with (\ref{alpha 3B}) implies $\langle \alpha, \alpha_1 \rangle = \pm p$, and so using (\ref{alpha 1B}) and (\ref{alpha 2B}), we have  
\[m=2p/ \hat{a}_1=(4m-2p)/\hat{a}_2.\]

Using the second expression for $m$, we have $m= 2p/(4- \hat{a}_2)$, and then setting this equal to $2p/ \hat{a}_1$, gives $\hat{a}_1+\hat{a}_2=4$. Since we are assuming $m\geq p$, this leaves the two cases: $\hat{a}_1=1$, corresponding to the case when $m=2p$; and $\hat{a}_1=2$, corresponding to the case when $m=p$.  In both cases, i.e., when $m = 2p$ and $m = p$, we have that $\pm \alpha_2$ and $\pm (\alpha_1+\alpha_2)$ are possible screeners, and then since $2\alpha/\langle\alpha,\alpha\rangle\in L^\circ$, these are indeed screeners.

Now suppose $n_1=2$ or $-1$. Equation (\ref{alpha 3A}) gives $\langle \alpha,\alpha_1\rangle =\pm 3p$, and using (\ref{alpha 1A}), (\ref{alpha 2B}), and (\ref{alpha 3C}), we have
\[\langle \alpha,\alpha\rangle=6p= (4m-2p)/\hat{a}_2= 4p+m.\]
This gives $\hat{a}_2=1$, and $m=2p$. Therefore, $\pm(2\alpha_1+\alpha_2)$ and $\pm(\alpha_1-\alpha_2)$ are possible screeners when $m=2p$. Verifying that $2\alpha/\langle\alpha,\alpha\rangle\in L^\circ$, we see that these are indeed screeners.

Now assume $\langle\alpha,\hat{\alpha}_2\rangle = 2 (2m-p)$, i.e. $n_2=2$. But then (\ref{alpha 1A})--(\ref{alpha 2A}), and (\ref{alpha 3A}) imply $n_1=0, 2$.  But since $\pm \alpha_1$ is already a screener, the case $n_1 = 0$ contradicts Lemma \ref{basis screener results}$(i)$.  The case $n_1 = 2$, would give $\alpha = 2 \alpha_1 + 2 \alpha_2$ as a screener with (\ref{alpha 1A}) and (\ref{alpha 3C}) implying $m = p$.  But in this case $\pm (\alpha_1 + \alpha_2)$ has already been determined to be a screener, again contradicting Lemma \ref{basis screener results}$(i)$.  Thus this leads to no additional screeners.

 Now assume $\langle\alpha,\hat{\alpha}_2\rangle = 3 (2m-p)$, i.e. $n_2=3$. Then (\ref{alpha 1A})--(\ref{alpha 2A}),  (\ref{alpha 3A}), and (\ref{alpha 3C}) imply $n_1=0,1,2,3$  and 
 \[\langle\alpha,\alpha\rangle = 6(2m-p)= 9m + 2p(n_1^2-3n_1).\] 
For the cases when $n_1 = 1,2$, the equation above  implies $m=2p/3$, which contradicts the assumption that $m\geq p$.  For the cases when $n_1 = 0, 3$, the equation above implies that $m = 2p$ and $\alpha = 3 \alpha_2$ and $3(\alpha_1 + \alpha_2)$, respectively, which contradicts  Lemma \ref{basis screener results}$(i)$, since we have already determined that in this case $\alpha_2$ and $\alpha_1 + \alpha_2$ are screeners. Thus we get no new screeners here, and this completes the classification of screeners for Type 2.

To see that in the case 2(b), the set of screeners is the set of roots for $B_2 = C_2$, we note that the basis $\alpha_1$, $\alpha_2$ gives the Cartan matrix
\[\left( \begin{array}{rr}
2 & -2 \\
-1 & 2 
\end{array} \right) \]
giving the $B_2 = C_2$ root system.   And for the case 2(c), we see that the extra nonroot screeners in $D_2 = A_2$ give the $G_2$ root lattice via the correspondence $\alpha \longleftrightarrow  \alpha_1$ and $\beta \longleftrightarrow  \alpha_2 - \alpha_1$, for $\alpha, \beta$ a simple set of roots for $G_2$ with Cartan matrix
\[\left(\begin{array}{rr}
2 & -1\\
-3 & 2 
\end{array}
\right).
\] 
\end{proof}

\subsection{Classification of screeners for $L=\sqrt{p}K$ with $K$ a simply laced root lattice generated by screeners}

First we prove the following lemma:

\begin{lem} \label{2p 4p 6p}
Let $L$ be a lattice generated by a system of screeners all of the same length, i.e., $L=\sqrt{p}K$, for $K$ a simply laced root lattice. Then every screener in $L$ has norm squared $2p$, $4p$, or $6p$.
\end{lem}

\begin{proof}
Let $L$ be as in the Lemma.  Then any $\alpha\in \sqrt{p}K$ has $\langle\alpha,\alpha\rangle\geq 2p$ and there exists a basis of screeners $S=\{\alpha_1,\cdots \alpha_d\}$ satisfying $\langle\alpha_i,\alpha_i\rangle =2p$. Let $\alpha$ be another screener. By positive definiteness, at least one of $\langle \alpha,\alpha_i\rangle$ is nonzero, and by Lemma \ref{interaction of bilinear form}, $\langle \alpha,\alpha\rangle= a \langle \alpha_i,\alpha_i\rangle= 2p a$ for $a=1,2,3$. 
\end{proof}

\begin{thm} \label{classification A}   For the root lattices $E_6$, $E_7$, $E_8$ and $A_n$ for $n>3$,  the set of all screeners is just the set of roots. 

The set of screeners for $A_3$ is the set of roots in addition to 6 screeners that are not roots of $A_3$ that satisfy $\langle \alpha,\alpha\rangle =4$ and are given by  $$\{\pm(\alpha_1+\alpha_3), \ \pm(\alpha_1-\alpha_3), \ \pm(\alpha_1+2\alpha_2+\alpha_3)\}.$$ 
Furthermore the set of roots along with nonroot screeners of $A_3$ give the roots of $C_3$.  
\end{thm}

\begin{proof}
First we prove the result for $A_d$ with $d\geq 3$. A $\Z$-basis of roots can be given so that $\langle\alpha,\alpha\rangle=2$, $\langle \alpha_i,\alpha_j\rangle=-1$ when $|i-j|=1$, and $\langle \alpha_i,\alpha_j\rangle=0$  otherwise whenever $i\neq j$. By Lemma \ref{2p 4p 6p}, any screener $\alpha$ that is not a root satisfies $\langle\alpha,\alpha\rangle=4,6$. Noting that Det$(G_{A_d})=d+1$, by Lemma \ref{divides deg G}, we know that if $\langle \alpha,\alpha\rangle =4$ then $d=2m+3$ for $m\geq 0$, while if $\langle \alpha,\alpha\rangle =6$, then $d= 3n+5$ for $n\geq 0$. We will observe below that no such elements exist unless $d=3$ and determine the non-root screeners when $d=3$.

Let $\alpha\in \Phi$ be a screener. If we set $\alpha= \sum_{i=1}^d n_i \alpha_i$, we can write $\langle\alpha,\alpha\rangle$ as a sum of squares:
\begin{align} \label{congruence relations for A_n 1}
\langle\alpha,\alpha\rangle& =(\sum_{i=1}^{d-1} 2n_i^2-2n_i n_{i+1})+ 2 n_d^2 \\
&=n_1^2 + \displaystyle\sum_{i=1}^{d-1} (n_{i+1}-n_i)^2 + n_d^2. \nonumber  
\end{align}
Thus $\alpha\in L$ such that
$\langle\alpha,\alpha\rangle=4$ must satisfy 
\begin{align} \label{alpha norm 4}
n_1,n_i-n_{i-1},n_d&=0,\pm 1, \quad \mathrm{and} \nonumber  \\
|n_1|,|n_i-n_{i-1}|,|n_d|&=1 \mbox{ exactly 4 times} \noindent 
\end{align} 
for  $i=2,\dots,d$, while $\alpha$ such that
$\langle\alpha,\alpha\rangle=6$ must satisfy 
\begin{align} \label{alpha norm 6}
n_1,n_i-n_{i-1},n_d&=0,\pm 1,\pm 2
\end{align}
for  $i=2,\dots,d$, and either
\begin{align}\label{alpha norm 6 1} 
|n_1|,|n_i-n_{i-1}|,|n_d|&=1  &\mbox{ exactly 6 times,}\\  
& \mbox{ or } &  \nonumber  \\
|n_1|,|n_i-n_{i-1}|,|n_d|&=2 &\mbox{ exactly 1 time} \label{alpha norm 6 2}\\ 
& \mbox{ and }& \nonumber \\
\label{alpha norm 6 3} \quad  |n_1|,|n_i-n_{i-1}|,|n_d|& =1 & \mbox{ exactly 2 times}   .  
\end{align}

By Lemma \ref{interaction of bilinear form}, we must also have that
\begin{align}
\langle\alpha,\alpha_1\rangle&= 2n_1 - n_2 &=  0,\pm t \label{congruence relations for A_n 2}\\
\langle\alpha,\alpha_j\rangle&=-n_{j-1} + 2n_j-n_{j+1} & =0,\pm t\label{congruence relations for A_n 3} \\
\langle\alpha,\alpha_d\rangle&= 2n_{d}- n_{d-1}&=0,\pm t  \label{congruence relations for A_n 4}
\end{align}
for $2\leq j\leq d-1$ and either $t=2$ or $t=3$ corresponding to $\langle\alpha,\alpha\rangle =4$ or $\langle\alpha,\alpha\rangle =6$, respectively.

Now suppose $n_1\in t\Z$ or $n_d\in t\Z$. The relations (\ref{congruence relations for A_n 2})--(\ref{congruence relations for A_n 4}) imply $n_2,\dots, n_d\in t\Z$ or $n_1,\dots, n_{d-1} \in t\Z$, implying $\alpha\in tL$, which contradicts Lemma \ref{basis screener results}$(i)$. So both $n_1$ and $n_d$ are $\pm1$ when $t=2$, and $\pm 1,\pm 2$ when $t=3$.

When $\langle\alpha,\alpha\rangle =4$, (\ref{congruence relations for A_n 2})--(\ref{congruence relations for A_n 4}) also imply that for $2\leq k\leq d-1$, $n_k\in 2\Z+i$ for $k\in 2\Z+i$, for $i=0,1$. Together with (\ref{alpha norm 4}), the only possibilities for screeners are thus $\pm(\alpha_1\pm \alpha_3)$ and $\pm(\alpha_1+2\alpha_2+\alpha_3)$ when $d=3$.

Now assume $\langle\alpha,\alpha\rangle =6$.   Suppose $\alpha$ satisfies (\ref{alpha norm 6 1}). 
Then we must have $d\geq 5$ and $n_1=\pm 1$, and we can replace $\alpha$ with $-\alpha$ so that $n_1$ is positive. Then using (\ref{congruence relations for A_n 2}), $n_2=2$, and using (\ref{congruence relations for A_n 3}), $n_3=3$, $n_4=4$, which already gives $\langle\alpha,\alpha\rangle >6$ in (\ref{congruence relations for A_n 1}). Suppose $\alpha$ satisfies (\ref{alpha norm 6 2}) and (\ref{alpha norm 6 3}).  If $n_1=2$, then $n_2=1$ and $n_3=0=n_4=\dots= n_d$ in order for (\ref{alpha norm 6 2}) and (\ref{alpha norm 6 3}) to be satisfied when $d>2$, contradicting $n_d\neq 0$. Similarly when $n_1=1$, using (\ref{congruence relations for A_n 2}) we have $n_2=2$ or $-1$, and then $n_3=0=n_4=\dots= n_d=0$ in order for  (\ref{alpha norm 6 2}) and (\ref{alpha norm 6 3}) to be satisfied when $d>2$, again contradicting $n_d\neq0$.  We can conclude there are no screeners satisfying  $\langle\alpha,\alpha\rangle =6$ when $d>2$. Therefore the only possibilities that satisfy (\ref{alpha norm 6}) and (\ref{congruence relations for A_n 2})--(\ref{congruence relations for A_n 4}) are $\pm(\alpha_1+2\alpha_2),\pm(2\alpha_1+\alpha_2), \pm(\alpha_1-\alpha_2)$ when $d=2$ as in Theorem \ref{classification rank 2}.

The set of roots for $A_3$ and nonroot screeners give the roots of $C_3$ via taking as the simple roots, the two roots $\alpha = \alpha_2$, $\beta = \alpha_1$ of norm squared 2,  and $\gamma = \alpha_3 - \alpha_1$  of norm squared 4, which give the Cartan matrix
\[ \left( \begin{array}{rrr}
2 & -1 & 0\\
-1 & 2 & -1 \\
0 & -2 & 2 
\end{array}
\right).\]

Next we consider the case of $E_6$.  A $\Z$-basis $\{\alpha_1,\cdots,\alpha_6\}$ can be chosen for $E_6$ with Gram matrix
$$G=\begin{pmatrix} 
2 & -1 & 0 & 0 & 0 & 0\\ 
-1 & 2 &-1 & 0 & 0  & 0 \\
0 & -1 & 2 & -1 &  0 & -1  \\
0 & 0 & -1 & 2 &  -1 & 0  \\
0 & 0 & 0 & -1 &  2  &0 \\
0 & 0  &-1& 0 &  0&2 \
\end{pmatrix}.$$
Since $\mbox{Det}(G)=3$, $E_6$ can have screeners of norm squared $2$ or $6$. We will show there are no screeners $\alpha$ satisfying  $\langle\alpha,\alpha\rangle =6$. 

Suppose
$\alpha=\sum_{i=1}^6 n_i\alpha_i$ is a screener with  $\langle\alpha,\alpha\rangle =6$. Then we have that
\begin{multline} \label{E6 norm}
\langle\alpha,\alpha\rangle = n_1^2+(n_1-n_2)^2+ (n_2-n_3)^2+ (n_4-n_3)^2\\
 + (n_5-n_4)^2 + (n_6-n_3)^2 + n_5^2+n_6^2-n_3^2
\end{multline} 
subject to
\begin{align}
&\langle \alpha_1,\alpha\rangle = 2n_1-n_2 &= 0,\pm 3 \ \label{E6 1} \\
&\langle \alpha_2, \alpha\rangle = -n_{1}+2n_2- n_{3} &= 0,\pm 3 \ \label{E6 2} \\
&\langle \alpha_3, \alpha\rangle = -n_{2}+2n_3- n_{4}-n_6 &= 0,\pm 3 \  \label{E6 3} \\
&\langle \alpha_4, \alpha\rangle =-n_3+ 2 n_4-n_5 &= 0,\pm 3 \ \label{E6 4}\\
&\langle \alpha_5, \alpha\rangle= 2 n_5-n_4 &= 0,\pm 3 \ \label{E6 5} \\ 
&\langle \alpha_6, \alpha\rangle = 2 n_6-n_3  & = 0,\pm 3 \label{E6 6}.
\end{align}

Suppose that $n_3\in 3\Z+i$ for $i=0,1,2$. Then using (\ref{E6 2}), we have $2n_2-n_1\in 3\Z+i$. Using (\ref{E6 1}), $2n_1-n_2\in 3\Z$. Adding these together gives $n_1+n_2\in 3\Z+i$, and then adding this to $2n_2-n_1$ gives  $3n_2\in 3\Z+2i$, so it must be that $n_3\in 3\Z$. 

Fix $n_3=3i$ for $i\in \Z$, and replace $\alpha$ with $-\alpha$ if necessary so that $n_3 $ is nonnegative. 

Viewing (\ref{E6 norm}) as a function in the real variables $n_1,n_2,n_4,n_5,n_6$, we have that (\ref{E6 norm}) has a global minimum of $\langle\alpha,\alpha\rangle \geq 3 i^2/2$ at $n_1=i,\,n_2=2i,\, n_4=2i,\, n_5=i$, and $n_6 = 3i/2$. Since we want $\langle\alpha,\alpha\rangle =6$, this leaves the cases $i=0,1,2$. When $i=2$, we have $\langle\alpha,\alpha\rangle \geq 6$, with equality if and only if (\ref{E6 norm}) takes on the global minimum with $(n_1,\dots,n_6)=(2,4,6,4,2,3)$, but in this case (\ref{E6 3}) does not hold. When $i=1$, (\ref{E6 6}) gives $n_6=0$ or $3$. In either case, this gives a value greater than $\langle\alpha,\alpha\rangle =6$ for (\ref{E6 norm}), thus not producing any new screeners.

We have left to check the case $i=0$ and thus $n_3=0$. In this case, the relations  (\ref{E6 norm}) and  (\ref{E6 1})--(\ref{E6 6}) are equivalent to the relations (\ref{congruence relations for A_n 1}) and (\ref{congruence relations for A_n 2})--(\ref{congruence relations for A_n 4}) for $A_5$ with $n_3=0$ under the bases determined by the Gram matrices prescribed. Since we have already shown that $A_5$ has no screeners that were not roots, in particular when $n_3=0$, neither can $E_6$. Thus no screeners $\alpha$ with $\langle\alpha,\alpha\rangle =6$ exist for $E_6$.

Next we consider the case of $E_7$.  A $\Z$-basis $\{\alpha_1,\cdots,\alpha_6\}$ can be chosen for $E_7$ with Gram matrix
$$G=\begin{pmatrix} 
2 &-1 &0&0&0&0&0 \\
-1&2 & -1 & 0 & 0 & 0 & 0\\ 
0&-1 & 2 &-1 & 0 & 0  & 0 \\
0&0 & -1 & 2 & -1 &  0 & -1  \\
0&0 & 0 & -1 & 2 &  -1 & 0  \\
0&0 & 0 & 0 & -1 &  2  &0 \\
0&0 & 0  &-1& 0 &  0&2 \
\end{pmatrix}.$$
Since $\mbox{Det}(G)=2$, $E_7$ can have screeners of norm squared $2$ or $4$. We show there are no screeners $\alpha$ satisfying  $\langle\alpha,\alpha\rangle =4$.

Suppose
$\alpha= \sum_{i=1}^6 n_i \alpha_i$ is a screener with  $\langle\alpha,\alpha\rangle =4$. Then we have that
\begin{multline} \label{E7 norm}
\langle\alpha,\alpha\rangle
= n_1^2+(n_1-n_2)^2+ (n_2-n_3)^2+(n_4-n_3)^2 \\
+(n_5-n_4)^2 + (n_6-n_5)^2+ (n_7-n_4)^2 + n_7^2+n_6^2-n_4^2 
\end{multline} 
subject to
\begin{align}
&\langle \alpha_1,\alpha\rangle = 2n_1-n_2&=& 0,\pm 2  \label{E7 1} \\
&\langle \alpha_2, \alpha\rangle = -n_{1}+2n_2- n_{3}&=& 0,\pm 2 \label{E7 2} \\
&\langle \alpha_3, \alpha\rangle = -n_{2}+2n_3- n_{4}&=& 0,\pm 2 \label{E7 3} \\
&\langle \alpha_4, \alpha\rangle = -n_{3}+2n_4- n_{5}- n_{7}&=& 0,\pm 2 \label{E7 4} \\
&\langle \alpha_5, \alpha\rangle =-n_4+ 2 n_5-n_6 &=& 0,\pm 2 \label{E7 5}\\
&\langle \alpha_6, \alpha\rangle= 2 n_6-n_5 &=& 0,\pm 2 \label{E7 6} \\ 
&\langle \alpha_7, \alpha\rangle = 2 n_7-n_4 &=& 0,\pm 2 \label{E7 7}.
\end{align}

Equations (\ref{E7 6}) and (\ref{E7 7}) imply $n_4,n_5\in 2\Z$, and then (\ref{E7 3}) and (\ref{E7 5}) imply $n_2,n_6\in 2\Z$. Also (\ref{E7 2}) and (\ref{E7 4}) then imply that either all of $n_1,n_3,n_7\in 2\Z$ or all of $n_1,n_3,n_7\in 2\Z+1$. The former would imply $\alpha\in 2L$, so we must have the latter. However, if $n_1,n_3,n_7\in 2\Z+1$ while $n_2,n_4,n_5,n_6\in 2\Z$, then (\ref{E7 norm}) has at least 6 odd terms added together, namely  $$n_1^2+(n_1-n_2)^2+ (n_2-n_3)^2+(n_4-n_3)^2+ (n_7-n_4)^2 + n_7^2$$ adding up to at least $6>4$, so that $\langle\alpha,\alpha\rangle\geq 6+n_6^2+(n_6-n_5)^2+(n_5-n_4)^2-n_4^2$. In particular we must have that $n_4\neq 0$.

Fix $n_4=2 i$ for $i\neq 0$, and replace $\alpha$ with $-\alpha$ if necessary so that $n_4 $ is nonnegative.  By (\ref{E7 7}) and the fact that $n_7\in 2\Z+1$, we have that $n_7=i$ if $i$ is odd and $n_7 = i\pm 1$ if $i$ is even. Thus 
\begin{multline}\label{E7}
 \langle\alpha,\alpha\rangle = n_1^2+(n_1-n_2)^2+ (n_2-n_3)^2+(2i-n_3)^2 \\
+(n_5-2i)^2 + (n_6-n_5)^2 +n_6^2 + 2\epsilon  -2i^2 
\end{multline} 
where $\epsilon=0$ if $i$ is odd and $\epsilon=1$ if $i$ is even.
 Viewing  (\ref{E7}) as a function in the real variables $n_1,n_2,n_3,n_5,n_6$, we observe that  (\ref{E7}) has a global minimum of $i^2/3+2\epsilon$ when $n_1=i/2$, $n_2=i$, $n_3=3i/2,$ $n_5= 4i/3$, and $n_6=2i/3$.   This minimum is greater than 4 unless $\epsilon = 0$ and $i$ is odd with $i=1,3$, or $\epsilon=1$ and $i$ is even with $i=2$.

Note also that since $n_1,n_3,n_7\in 2\Z+1$ and $n_2,n_4,n_5,n_6\in 2\Z$, we have 
\begin{equation}
\langle \alpha, \alpha \rangle = K + J + 2 \epsilon - 2 i^2, \quad \mathrm{with} \ K = \sum_{l = 1}^4 k_l^2, \ \mathrm{and} \ J = \sum_{l = 1}^3 j_l^2
\end{equation}
where $k_l$ are odd for $l = 1, \dots, 4$ and $j_l$ are even for $l = 1,2,3$.  

Thus if $\langle \alpha, \alpha \rangle = 4$, we need $K+J = 4 - 2\epsilon + 2i^2$ which is 6, 22 or 10 in the three cases we need to consider of $\epsilon = 0$ and $i=1,3$, or $\epsilon=1$ and $i=2$.  However a straightforward calculation shows that the values of $K$ are 4, 12, 20, 28, .... and the values of $J$ are 0, 4, 8, 12, ..... leading to no values of $K + J$ of 6, 10, or 22.  Hence there are no screeners satisfying $\langle\alpha,\alpha\rangle=4$ for $E_7$, and all screeners must be roots.

Finally, since $E_8$ has $\mbox{Det}(G)=1$, all screeners $\alpha$ are roots satisfying $\langle\alpha,\alpha\rangle=2$, and this completes the proof. 
\end{proof}

Next we will look at $D_n$ for $n>3$. In Theorem \ref{classification rank 2} and Theorem \ref{classification A}, we saw $D_2 = A_2$ and $D_3 = A_3$ had $6$ screeners that are not roots and are of norm squared $6$ and $4$ respectively. For $D_n$, $n\geq 3$, we will use a $\Z$-base with Cartan matrix
$$\begin{pmatrix} 
2 & 0 & -1 & 0 & 0 &\cdots & 0 & 0 &0 & 0 \\ 
0 & 2 &-1 & 0 & 0  & \cdots & 0  & 0 &0 & 0  \\
-1 & -1 & 2 & -1 &  0  &  \cdots  & 0 & 0 &0 & 0 \\
0 & 0 & -1 & 2 &  -1  &  \cdots  & 0 & 0 &0 & 0\\
0 & 0 & 0 & -1 &  2   &  \cdots  & 0 & 0 &0 & 0 \\
\vdots & \vdots & \vdots & \vdots & \vdots &  \ddots& \vdots & \vdots& \vdots & \vdots \\
0 & 0 & 0 & 0 & 0 &  \cdots  & 2 & -1 & 0 &0 \\
0 & 0 & 0 & 0 & 0 &  \cdots & -1  & 2 & -1 & 0 \\
0 & 0 & 0 & 0& 0 &  \cdots  & 0 & -1 & 2 & -1 \\
0 & 0  & 0 & 0& 0 &  \cdots  & 0 &0 & -1 & 2  \\
\end{pmatrix}.$$

\begin{thm}\label{classification-Dn} The root lattice $D_n$ for $n\geq 2$ has additional screeners which are not roots for $D_n$, but give root systems for non-simply laced Lie algebras.  $D_2$ has additional screeners of norm squared $6$ giving rise to the $G_2$ roots, while $D_n$ for $n\geq 3$ has additional screeners of norm squared 4, giving rise to the $F_4$ roots and the $C_n$ roots for $n = 4$ and $n >4$, respectively.  In particular, with respect to the $\Z$-basis used in the Gram matrix above, these additional screeners are given by: \\

	(i) For $D_2 = A_2$ there are 6 nonroot screeners given by 
\[ \{\pm(\alpha_1-\alpha_2), \ \pm(\alpha_1+2\alpha_2), \ \pm(2\alpha_1+\alpha_2)\}\]
extending the roots of $A_2$ to include the roots of $G_2$. 

	(ii) For $D_3=A_3$ there are $6$ nonroot screeners given by
 \[ \{\pm(\alpha_1+\alpha_2), \ \pm(\alpha_1-\alpha_2), \ \pm(\alpha_1+\alpha_2+2\alpha_3) \},\]
extending the roots of $A_3$ to include the roots of $C_3$.

	(iii) For $D_4 $ there are $24$ nonroot screeners given by 
\begin{multline*}
\{\pm( \alpha_i+ \alpha_j), \ \alpha_i- \alpha_j, \ \pm(2\alpha_i+\alpha_j+2\alpha_3+\alpha_k),\\     
\pm(\alpha_i+2\alpha_3+\alpha_j) \, |\, \{i,j,k\}= \{1,2,4\}\},
\end{multline*}
extending the roots of $D_4$ to include the roots of $F_4$.  
 
	(iv) For $D_n$, for $n>4$, there are $2n$ nonroot screeners given by
 \[\{\pm(\alpha_1+\alpha_2), \ \pm(\alpha_1-\alpha_2),\ \pm(\alpha_1+\alpha_2+2\sum_{i=3}^j \alpha_i) \,| 3\leq j\leq n  \},\] 
extending the roots of $D_n$ to the roots of $C_n$.  
\end{thm} 

\begin{proof} The screeners that are not roots of $D_n$ in the cases for when $n = 2$ and $n = 3$ have already been proven in Theorems \ref{classification rank 2} and \ref{classification A}, including the fact that $D_3 = A_3$ gives rise to $C_3$.

As seen in Theorem \ref{classification rank 2}, the extra nonroot screeners in $D_2 = A_2$ give the $G_2$ root lattice via the correspondence $\alpha \longleftrightarrow \alpha_1$ and $\beta \longleftrightarrow \alpha_2 - \alpha_1$, for $\alpha, \beta$ a simple set of roots for $G_2$ with Cartan matrix
\[\left(\begin{array}{rr}
2 & -1\\
-3 & 2 
\end{array}
\right).
\] 

By Lemma \ref{2p 4p 6p}, any screener $\alpha$ that is not a root satisfies $\langle\alpha,\alpha\rangle=4$ or $6$. Note that $\mbox{Det}(G_{D_d})=4$ for $d> 3$, so we only need to check for screeners $\alpha$ satisfying $\langle\alpha,\alpha\rangle=4$.  

We can write $\langle\alpha,\alpha\rangle$ as the quadratic form
\begin{align} \label{relation D}
\langle\alpha,\alpha\rangle= n_1^2 +n_2^2 + n_d^2 +\displaystyle\sum_{j=2}^{d-1} (n_{j+1}-n_j)^2+ (n_{3}-n_1)^2-n_3^2.  
\end{align}
In addition, $\alpha$ must satisfy the relations
\begin{align}
\langle \alpha_1,\alpha\rangle& = 2n_1 - n_3 & =& 0,\pm 2 \label{relation D1}\\
\langle \alpha_2,\alpha\rangle &=2n_2 - n_3 & =& 0,\pm 2 \label{relation D2}\\
\langle \alpha_3,\alpha\rangle &=2n_3-n_1-n_2-n_4 &=& 0,\pm 2 \label{relation D3} \\
\langle \alpha_j,\alpha\rangle &=-n_{j-1} + 2n_j-n_{j+1} &=& 0,\pm 2 \label{relation D4}\\
\langle \alpha_d,\alpha\rangle &=2 n_d- n_{d-1}  & =& 0,\pm 2, \label{relation D5}
\end{align}
for $4\leq j\leq  d-1.$ Then  (\ref{relation D1}) gives $n_3\in 2\Z$.  As usual we can replace $\alpha$ with $-\alpha$ so that $n_3$ is nonnegative and write $n_3=2i$. 

Suppose $d=4$. Then (\ref{relation D}) has a global minimum of  $2i^2$ at $n_1=n_2=n_4=i$, implying $n_3$ is either $0$ or $2$.  Using relation (\ref{relation D3}), at least one of $n_1,n_2,n_4$ or all three of must be in $2\Z$. However, not all three can be in $2\Z$, or else $\alpha\in 2L$. When $n_3=0$, this leaves $\alpha_i\pm \alpha_j$ for $i,j\in \{1,2,4\}$ with $i\neq j$, while $n_3=2$ leaves $\alpha_i+\alpha_j+2\alpha_3 +2\alpha_k$ and $\alpha_i+\alpha_j+2\alpha_3$ for $i,j,k\in \{1,2,4\}$ with i$\neq j\neq k$, giving the list of extra screeners when $d = 4$.

To see that this gives the $F_4$ root lattice, we consider $\beta_1 = \alpha_4-\alpha_1$, $\beta_2 = \alpha_1 - \alpha_2$, $\beta_3 = \alpha_2$ and $\beta_4 = \alpha_3$ and note that this set of simple roots give the $F_4$ Cartan matrix 
\[\left(\begin{array}{rrrr}
2 & -1 & 0 & 0\\
-1 & 2 & -2 & 0\\
0 & -1 & 2 & -1\\
0 & 0 & -1 & 2
\end{array}
\right).
\]

For $d>4$, Equation (\ref{relation D4}) gives $n_{2j+1}\in 2\Z$ for all $j\geq 2$, and (\ref{relation D4}) also implies either $n_{2j}\in2\Z$ for all $ j\geq 2$ or $n_{2j}\in2\Z+1$ for all $ j\geq 2$. But relation (\ref{relation D5}) implies that $n_{d-1}\in 2\Z$. Thus $n_{2j}\in 2\Z+1$ for all $j\geq 2$ can only hold if $d$ is even.

In addition, by (\ref{relation D1}) and (\ref{relation D2}), we have that $n_1 = \epsilon_1 + i$ and $n_2 = \epsilon_2 + i$ for $\epsilon_1, \epsilon_2 \in \{0, \pm 1\}$.  Thus
\begin{equation}\label{D again}
n_1^2 +n_2^2 + (n_1-n_3)^2+ (n_2-n_3)^2-n_3^2 = 2 \epsilon_1^2 + 2 \epsilon_2^2 .
\end{equation}

Suppose $n_{2j}\in 2\Z+1$ for all $j\geq 2$. Then using (\ref{relation D3}), $n_1+n_2\in 2\Z+1$ since $n_4\in 2\Z+1$. So exactly one of $n_1$ and $n_2$ is odd.  Thus $2 \epsilon_1^2 + 2 \epsilon_2^2 = 2$.  However, $n_d^2 + \sum_{j=3}^{d-1} (n_{j+1}-n_j)^2\geq d-2$ since there are $d-2$ nonzero terms.  We also have that $d$ is even and greater than 4, and thus in particular, $d\geq 6$.  It follows from this and Equation (\ref{relation D}) that $\langle \alpha,\alpha\rangle \geq (d-2) + 2  = d >4$, implying that there are no screeners satisfying $\langle \alpha, \alpha \rangle = 4$ in this case.

Suppose $n_j\in 2\Z$ for $j\geq 4$.  Since $n_3$ is also even, if $n_1,n_2\in 2\Z$, then $\alpha \in 2L$, and so both $n_1,n_2\in 2\Z+1$.    Thus if $i$ is even, $\epsilon_1, \epsilon_2 \in \{\pm 1\}$ and if $i$ is odd, $\epsilon_1 = \epsilon_2 = 0$.

Therefore if $i$ is even, Equation (\ref{relation D})  implies that $\langle \alpha, \alpha \rangle = 4 + n_d^2 + \sum_{j=3}^{d-1} (n_{j+1}-n_j)^2$, and thus we must have $n_d = n_{d-1} = \cdots = n_3 = 0$, giving $n_1, n_2 \in \{\pm 1\}$, and the screeners $\pm(\alpha_1 + \alpha_2), \ \pm(\alpha_1 - \alpha_2)$. 

If $i$ is odd,  we have $\langle \alpha, \alpha \rangle = n_d^2 + \sum_{j=3}^{d-1} (n_{j+1}-n_j)^2$, and we need 1 and only 1 of the summands, which are all even, to be nonzero and in fact $4$.  But this implies $i=1$ and $n_3=n_4=\cdots n_r=2$, for $3\leq r\leq d$, and $n_{r+1}=n_{r+2}=\cdots = n_d=0$. This gives rise to screeners of the form $\pm (\alpha_1+\alpha_2+ 2 \sum_{j=3}^r\alpha_j)$ for $3\leq r\leq d$ as the only possible screeners with $\langle\alpha,\alpha\rangle=4$ when $n_3=2$.  This gives a total of $2d$ screeners that are not roots for $d>4$. 

The fact that these roots of $D_n$ and additional nonroot screeners for $n>4$ give rise to $C_n$ can be seen by taking the simple roots to be  $\alpha_1 - \alpha_2, \alpha_2, \alpha_3, \dots, \alpha_{n-1}, \alpha_n$ with the Cartan matrix then given by
\begin{equation}\label{Cn-Cartan}
\begin{pmatrix} 
2 & -2 & 0 & 0 & 0 &\cdots & 0 & 0 &0 & 0 \\ 
-1 & 2 &-1 & 0 & 0  & \cdots & 0  & 0 &0 & 0  \\
0 & -1 & 2 & -1 &  0  &  \cdots  & 0 & 0 &0 & 0 \\
0 & 0 & -1 & 2 &  -1  &  \cdots  & 0 & 0 &0 & 0\\
0 & 0 & 0 & -1 &  2   &  \cdots  & 0 & 0 &0 & 0 \\
\vdots & \vdots & \vdots & \vdots & \vdots &  \ddots& \vdots & \vdots& \vdots & \vdots \\
0 & 0 & 0 & 0 & 0 &  \cdots  & 2 & -1 & 0 &0 \\
0 & 0 & 0 & 0 & 0 &  \cdots & -1  & 2 & -1 & 0 \\
0 & 0 & 0 & 0& 0 &  \cdots  & 0 & -1 & 2 & -1 \\
0 & 0  & 0 & 0& 0 &  \cdots  & 0 &0 & -1 & 2  \\
\end{pmatrix},
\end{equation}
finishing the proof. 
\end{proof}

\subsection{Summary for positive definite even lattices generated by a screening system}

We now summarize the previous results above in Theorems \ref{classification A} and \ref{classification-Dn}, as well as the cases in Theorem \ref{classification rank 2} for a lattice generated by a screening system, and the implications of Theorem 5.11.

\begin{thm}  Let $L$ be a positive definite even lattice that is generated by a screening system.  Then the set of all screeners for $L$ consists of the roots in an orthogonal direct sum of the following (in general rescaled) root systems:  $A_1$, $B_2$, $G_2$, $B_3$, $C_3$, $B_4$, $F_4$, $E_6$, $E_7$, $E_8$, and $A_n$, $C_n$,  and $B_n$, for $n >4$. 
\end{thm}

\begin{proof} From Theorems \ref{classification A}, \ref{classification-Dn}, all the cases in the statement of the Corollary appear except for the cases when additional screeners arise from rescaled orthogonal simply laced lattices.  In Theorem \ref{classification rank 2}, the case $A_1$ x $A_1$ is treated and shown to give rise to $C_2 = B_2$.  Thus in the list above only $B_n$, for $n \geq 4$ is not included in our previous theorems.  However, the case $(A_1)^2$ as treated in Theorem \ref{classification rank 2}, extends easily to $(A_1)^n$ and gives rise to additional nonroot screeners that are the roots of $B_n$, for $n > 1$.   That is, if $\alpha_1, \dots, \alpha_n$ are simple roots for $(A_1)^n$, then $\alpha_1, \alpha_2 - \alpha_1, \alpha_3- \alpha_2, \dots, \alpha_n - \alpha_{n-1}$ are a list of simple roots giving rise to the dual Cartan matrix to that of $C_n$ given in (\ref{Cn-Cartan}), and then the full set of screeners is given by the roots of $B_n$ realized as the $2n$ short roots $\pm \alpha_1, \dots, \pm \alpha_n$ and the $2n(n-1)$ long roots $\pm (\alpha_i + \alpha_j)$, and $\alpha_i - \alpha_j$ for $i,j \in \{ 1,\dots,n\}$.

By Theorem 5.11, it follows that only the rank 1 case of orthogonal screeners gives rise to new screeners, the others arising within the simply laced bases of screeners constructed.  
\end{proof}

{}

\begin{thebibliography}{FFHST}
	
\bibitem[Ab]{Abe} T. Abe, A $\mathbb{Z}\sb 2$-orbifold model of the symplectic fermionic vertex operator superalgebra. {\em Math. Z.}, {\bf 255} (2007), 755--792.

\bibitem[Ad]{A-2003}  D. Adamovi\'{c}, Classification of irreducible modules of certain subalgebras of free boson vertex algebra, {\it J. Alg.} \textbf{270} (2003), 115--132.
	
\bibitem[AM1]{AdM} D. Adamovi\'c and A. Milas, Logarithmic intertwining operators and $\mathcal{W}(2,2p-1)$-algebras, {\em J. of Math. Phys.} {\bf 48} 073503 (2007), 20 pp.
	
\bibitem[AM2]{am1} D. Adamovi\' c and A. Milas, On the triplet vertex algebra $W(p)$, {\it Adv. Math.} \textbf{217} (2008), 2664--2699.

\bibitem[AM3]{Am5}  D. Adamovi\' c and A. Milas, Lattice construction of logarithmic modules for certain vertex algebras,  {\em Selecta Math. (N.S.)}, \textbf{15} (2009), 535--561.

\bibitem[AM4]{am-sigma} D. Adamovi\' c and A. Milas, The $N=1$ triplet vertex operator superalgebras, {\em Comm. in Math. Phys.} \textbf{288} (2009), 225--270.

\bibitem[AM5]{am-imrn} D. Adamovi\' c and A. Milas, On $W$-algebras associated to $(2,p)$ minimal models and their representations, \textit{Int. Math. Res. Not.} \textbf{20} (2010), 3896--3934.
	
\bibitem[AM6]{am2} D. Adamovi\' c and A. Milas, The structure of Zhu's algebras for certain ${\mathcal W}$-algebras, \textit{Adv. in Math.} \textbf{227} (2011), 2425--2456.

\bibitem[AM7]{am8}  D. Adamovi\' c and A. Milas, On $W$-algebra extensions of $(2,p)$ minimal models: $p > 3$ , {\em J. of Alg.}  {\bf 344} (2011), 313--332.

\bibitem[AM8]{am9}  D. Adamovi\' c and A. Milas, An explicit realization of logarithmic modules for the vertex operator algebra $\mathcal{W}_{p,p'}$, {\it J. Math. Phys.} \textbf{53}, (2012),  16 pp.

\bibitem[AM9]{am10}  D. Adamovi\' c and A. Milas, The doublet vertex operator algebra $\mathcal{A}(p)$ and $\mathcal{A}_{2,p}$, in ``Recent Developments in Algebraic and Combinatorial Aspects of Representation Theory, {\em Contemp. Math.} {\bf 602} (2013), Amer. Math. Soc., Providence, RI, 23--38.
	
\bibitem[AM10]{am4}  D. Adamovi\' c and A. Milas, $C_2$-Cofinite $\mathcal{W}$-algebras and their logarithmic representations, in  ``Conformal Field Theories, and Tensor Categories", {\it Math. Lect. Peking Univ.,} Springer, Heidelberg, (2014) 249--270.

\bibitem[BV1]{BV1} K. Barron and N. Vander Werf, On permutation-twisted free fermion vertex operator superalgebras and two conjectures, in: ``Proceedings of the XXIst International Conference on Integrable Systems and Quantum Symmetries", June 2013, Prague, Czech Republic; ed. C. Burdik, O. Navratil and S. Posta; {\it J. of Physics: Conference Series} \textbf{474} (2013), 012009; 35 pp.

\bibitem[BV2]{BV-W-algebras} K. Barron and N.  Vander Werf, Classification of irreducible modules for the kernel of a screening operator for rank 2 lattice vertex operator algebras, in preparation.

\bibitem[BVY]{BVY} K. Barron, N. Vander Werf, and J. Yang, Higher level Zhu algebras and modules for vertex operator algebras, to appear in {\it J. Pure and Applied Alg.}; preprint available at arXiv:1710.04767.
	
\bibitem[BT]{BT} J. de Boer and T. Tjin, Quantization and representation theory of finite $\mathcal{W}$-algebras, {\it Comm. in Math. Phys.} \textbf{158} (1993), 485516.

\bibitem[CGR]{CGR}, T. Creutzig, A. Gainutdinov, and I. Runkel, A quasi-Hopf algebra for the triple vertex operator algebras, arXiv: 1712.072601v1.
	
\bibitem[CM]{CM} T. Creutzig and A. Milas, Higher rank partial and false theta functions and representation theory, {\it Adv. Math.} {\bf 314} (2017), 203--227.
	
\bibitem[D]{D}  C. Dong, Vertex algebras associated with even lattices,  {\em J. of Alg.} {\bf 160} (1993), 245--265.

\bibitem[DL1]{DL1}  C. Dong and J. Lepowsky, A generalization of vertex operator algebra,  in ``Algebraic Groups and Generalizatons, Proc. 1991 Amer. Math Soc. Summer Research Institute" ed. W. Haboush and B. Parshall, {\it Proc. Sympos. Pure Math.}, Amer. Math. Soc., Providence, 1993.	

\bibitem[DL2]{DL}C. Dong and J. Lepowsky, Generalized Vertex Algebras and Relative Vertex Operators, {\it Progress in Math.} \textbf{112}, Birkha\"user, Boston, 1993.

\bibitem[DLM1]{DLM1} C. Dong, H. Li, and G. Mason, Certain associative algebras similar to $U (sl_2)$ and Zhu’s algebra $A(V_L)$, {\it  J. Alg.} \textbf{196} (1997), 532–-551.

\bibitem[DLM2]{DLM} C. Dong, H. Li and G. Mason, Vertex operator algebras and associative algebras, {\it J. Alg.} \textbf{206} (1998), 67--96.

\bibitem [DN]{DN} C. Dong and K. Nagatomo, Classification of irreducible modules for the vertex operator algebra $M(1)^{+}$, {\em J. of Alg.} \textbf{216} (1999),  384--404.

\bibitem[DF1]{DF1} V. Dotesenko and V. Fateev, Conformal algebra and multipoint correlation functions in two-dimensional statistical models, {\em Nucl. Phys. } \textbf{B 240} (1984) 312.  
	
\bibitem[DF2]{DF2} V. Dotesenko and V. Fateev, Four point correlation functions and the operator algebra in the two-dimensional conformal invariant theories with central charge $c<1$, {\it Nucl. Phys.} \textbf{B 251} (1985) 691.  
	
\bibitem[E]{Ebe} W. Ebeling, Lattices and Codes, 2nd edition, {\em Adv. Lect. in Math.},  Friedr. Vieweg \& Sohn, Braunschweig, 2002.
	
\bibitem[EFH]{EFH} W. Eholzer, M. Flohr, A. Honecker, R. H\"ubel, W. Nahm, and R. Varnhagen, Representations of ${\W}$-algebras with two generators and new rational models, {\it Nucl. Phys.} \textbf{B 383} (1992), 249--288.
	
\bibitem[FFr]{FFr} B. Feigin and E. Frenkel, A family of representations of affine Lie algebras, (Russian) {\it Uspekhi Mat. Nauk} {\bf 43} (1988),  227--228; translation in {\it Russian Math. Surveys} {\bf 43} (1988), 221?222.
	
\bibitem [FF]{FF} B. Feigin and D. B. Fuchs, Representations of the Virasoro algebra, in `` Representations of Lie groups and related topics", {\it Adv. Stud. Contemp. Math.} {\bf 7} (1990),  Gordon and Breach, New York.

\bibitem [FGST1]{FGST1} B. Feigin, A. Ga\u\i nutdinov, A. Semikhatov, and I. Yu Tipunin, I, The Kazhdan-Lusztig correspondence for the representation category of the triplet $W$-algebra in logarithmic conformal field theories, {\it Theoret. and Mat. Phys.} {\bf 148} (2006),  1210--1235.
	
\bibitem [FGST2]{FGST2} B. Feigin, A. Ga\u\i nutdinov, A. Semikhatov, and I. Yu Tipunin, Logarithmic extensions of minimal models: characters and modular transformations. {\em Nucl. Phys.}  {\bf B 757} (2006), 303--343.
	
\bibitem [FGST3]{FGST3} B. Feigin, A. Ga\u\i nutdinov, A.  Semikhatov, and I. Yu Tipunin, Modular group representations and fusion in logarithmic conformal field theories and in the quantum group center, {\em Comm. Math. Phys.} {\bf 265} (2006), 47--93.

\bibitem[FT]{FT} B. Feigin and I. Tipunin, Logarithmic CFTs connected with simple Lie
algebras, arXiv:1002.5047.
	
\bibitem [Fe]{Fel} G. Felder, BRST approach to minimal models, {\it Nucl. Phys.} \textbf{B 317} (1989), 215--236.
	
\bibitem [FFHST]{FFHST} J. Fjelstad, J. Fuchs, S. Hwang, A.M. Semikhatov, and I. Yu. Tipunin, Logarithmic conformal field theories via logarithmic deformations, {\it Nucl. Phys.} {\bf B 
	633} (2002), 379--413.
	
\bibitem[F1]{F1} M. Flohr, On modular invariant partition functions of conformal field theories with logarithmic operators, {\it Internat. J. Modern Phys.}  {\bf A 11 } (1996), 4147--4172.
	
\bibitem[F2]{F2} M. Flohr, Bits and pieces in logarithmic conformal field theory, in ``Proceedings of the School and Workshop on Logarithmic Conformal Field Theory and its Applications (Tehran, 2001)", {\it Internat. J. Modern Phys.}  {\bf A 18} (2003), 4497--4591.
	
\bibitem[FG]{FG} M. Flohr and M. Gaberdiel, Logarithmic torus amplitudes, {\em J. Phys.}  {\bf A 39} (2006), 1955--1967.

\bibitem[FB]{FB} E. Frenkel and D. Ben-Zvi, {\em Vertex algebras and algebraic curves}, Mathematical Surveys and Monographs, {\bf 88}, American Mathematical
	Society, Providence, RI, 2001.	
	
\bibitem[FHL]{FHL} I. Frenkel,  Y.-Z. Huang, and J. Lepowsky. On axiomatic approaches to vertex operator algebras and modules, {\it Mem. Amer. Math. Soc.}  {\bf 494} (1993).
	
\bibitem[FLM]{FLM} I.  Frenkel, J. Lepowsky, and A. Meurman, Vertex Operator Algebras and the Monster, {\em Pure and Applied Math.},  Academic Press, 1988.

		
\bibitem[FHST]{FHST} J. Fuchs, S. Hwang, A. Semikhatov, and I. Yu. Tipunin, Nonsemisimple fusion algebras and the Verlinde formula, {\it Comm. Math. Phys.} {\bf 247} (2004),  713--742.
	
\bibitem[G]{Ga} M. Gaberdiel, An algebraic approach to logarithmic conformal field theory, in  ``Proceedings of the School and Workshop on Logarithmic Conformal Field Theory and its Applications (Tehran, 2001)," {\it Internat. J. Modern Phys.}  {\bf A 18} (2003), 4593--4638.
	
\bibitem[GK1]{GK1} M. Gaberdiel and H. G. Kausch, A rational logarithmic conformal field theory, {\it Phys. Lett} {\bf B 386} (1996), 131--137.
	
\bibitem[GK2]{GK2} M. Gaberdiel and H. G. Kausch, A local logarithmic conformal field theory, {\it Nucl. Phys.} {\bf B 538} (1999), 631-658.
	
\bibitem[GR]{GR} M. Gaberdiel and I. Runkel, The logarithmic triplet theory with boundary, {\em J.Phys.}  {\bf A 39} (2006), 14745--14780.

\bibitem[H]{H} J. Humphreys, {\it Introduction to Lie Algebras and Representation Theory}, Graduate Texts in Math., Vol. 9, Springer, New Your, 1972.
	
\bibitem [Ka1]{Ka1} H.  Kausch, Extended conformal algebras generated by multiplet of primary fields, {\it Phys. Lett.} {\bf B 259} (1991), 448--455.
	
\bibitem[Ka2]{Ka2} H. Kausch, Symplectic fermions, {\it Nucl. Phys.}  {\bf B 583} (2000), 513--541.
	
\bibitem[KW]{KaW} H. Kausch and G. Watts, A study of ${\W}$--algebras by using Jacobi identities, {\it Nucl. Phys.}  {\bf B 354} (1991), 740--768.

\bibitem[Le]{Lentner} S. Lentner, Quantum groups and Nichols algebras acting on conformal field theories, arXiv:1702.06431v1.
	
\bibitem[LL]{LL} J. Lepowsky and H. Li, {\em Introduction to Vertex Operator Algebras and Their Representations}, Progress in Mathematics, Vol. 227, Birkh\"auser, Boston, 2003.

\bibitem[Li]{Li} W. Li, Abelian intertwining algebras and modules related to rational lattices, {\it J. Alg.} {\bf 214} (1999), 356--384.

\bibitem[Ma]{MacLane} S. Mac Lane, Coholmology theory of abelian groups,  {\it Proc. Intl. Congress of Mathematicians, 1950}, Vol. II, 8--14. 

\bibitem[MP]{MP} A. Milas and M. Penn, Lattice vertex algebras and combinatorial bases: general case and $W$-algebras, {\it New York J. Math.} {\bf 18} (2012), 621--650.
	
\bibitem[Mi]{Miy} M. Miyamoto, Modular invariance of vertex operator algebras satisfying $C\sb 2$-cofiniteness. {\em Duke Math. J.} {\bf 122} (2004), 51--91.

\bibitem[NT]{NT} K. Nagatomo and A. Tsuchiya, The triplet vertex operator algebra $\mathcal{W}(p)$ and the restricted quantum group at root of unity, in ``Exploring new structures and natural constructions in mathematical physics", 149, {\it Adv. Stud. Pure Math.} {\bf 61}, Math. Soc. Japan, Tokyo, 2011.
	
\bibitem[N]{N} M. Newman, {\it Integral Matrices}, Academic Press, New York and London, 1972.
	
\bibitem[TW1]{TW1} A. Tsuchiya and S. Wood, The tensor structure on the representation category of the triplet algebra, {\it J. of Phys.} {\bf A 46} (2013), 40 pp.
	
\bibitem[TW2]{TW2} A. Tsuchiya and S. Wood, On the extended $W$-algebra of type $\mathfrak{sl}_2$ at positive rational level, {\it Inter. Math. Res. Not.} (2015), 5357--5435.
	
\bibitem[V]{V-thesis} N. Vander Werf, ``Screening Operators for Lattice Vertex Operator Algebras and Resulting Constructions", University of Notre Dame, 2017.

\bibitem[Wak]{Wakimoto} M. Wakimoto, Fock representations of affine Lie algebra $A^{(1)}_1$, {\it Comm. Math. Phys.} {\bf 104} (1986),  605--609.

\bibitem [Wan1]{W3}  W. Wang, Classification of irreducible modules of ${\W}_3$ algebra with  $c=-2$, {\it Comm. Math. Phys.} {\bf 195} (1998),  113--128.

\bibitem[Wan2]{W2} W. Wang, Nilpotent orbits and finite $W$-algebras, {\it Fields Inst. Commun.} {\bf 59} (2011), 71--105.
	
\bibitem[Xu]{Xu} X. Xu, {\it Introduction to Vertex Operator Superalgebras and Their Modules}, Kluwer Academic Publishers, Vol.  456, Boston, 1998.
	
\bibitem[Z]{Zhu} Y.-C. Zhu, Modular invariance of characters of vertex operator
	algebras, {\it J. Amer. Math. Soc.} {\bf  9} (1996), 237--302.
	
	
\end{thebibliography}
\end{document}